\documentclass[12pt,a4paper]{article}

\usepackage{geometry}
\usepackage{titling}
\usepackage{amsopn, amstext, amsmath,amsthm, amsbsy,mathrsfs}
\usepackage{latexsym,amssymb}
\usepackage{ragged2e}
\usepackage{url}
\usepackage{bbm}
\usepackage{authblk}
\usepackage{amsfonts}
\usepackage{dsfont}
\usepackage{color}
\usepackage{booktabs}
\usepackage{multirow}
\usepackage{array}

\usepackage{diagbox, fancyvrb}
\usepackage[english]{babel}

\usepackage{setspace} 
 
\usepackage{titlesec} 
\usepackage{titling}
\titleformat*{\section}{\normalsize\bfseries}
\titleformat*{\subsection}{\normalsize\it\bfseries}
\titlespacing*{\section}{0pt}{6pt}{6pt}

\titlelabel{\thetitle.\quad}

\RequirePackage{bm}

\numberwithin{equation}{section}
\usepackage[square,numbers,sort&compress]{natbib}

\usepackage[colorlinks,
linkcolor=blue,
anchorcolor=blue,
citecolor=blue
]{hyperref}

\newcommand{\sref}[2]{\hyperref[#2]{#1 \ref{#2}}}
\newcommand{\myeqref}[1]{\hyperref[#1]{\eqref{#1} }}

\DeclareMathOperator{\esssup}{esssup}
\newcommand\keywords[1]{\text{keywords}: #1 }
\newcommand{\as}{$\operatorname{\mathbb{P}-a.s.}$}
\newtheorem{theorem}{Theorem}[section]  
\newtheorem{Corollary}{Corollary}[section]    
\newtheorem{lemma}{Lemma}[section]  
\newtheorem{proposition}{Proposition}[section]

\date{\empty}
\renewenvironment{abstract}{\par\noindent\textbf{\abstractname}\ \ignorespaces}{\par\medskip}

\begin{document}

	\title{\large A Class of Multi-dimensional Backward Stochastic Differential Equations with Singular Generators exhibiting Diagonally Quadratic Growth and Applications}
			\author[a,$ \ast $]{Wenbo Wang} 
		\author[a,b]{Guangyan Jia}
		\affil[a]{\normalsize Zhongtai Securities Institute for Financial Studies,  Shandong University, Jinan, P.R. China}
		\affil[b]{\normalsize Shandong Province Key Laboratory of Financial Risk, Jinan, P.R. China}
		
		\vspace{0.8cm}
	\maketitle 		\thispagestyle{empty}
		\begin{spacing}{1.25}
		\begin{flushleft}
		{\small	
			\begin{abstract}
				
				\justifying \noindent
				This paper investigate a class of multi-dimensional backward stochastic differential equations (BSDEs) with  singualr generators exhibiting  diagonally quadratic growth and unbounded terminal conditions, thereby extending results in the literature. We present an example of such equations in optimal investment decision.		
			\end{abstract}

}
\end{flushleft}
\end{spacing}

\section{Introduction}

This paper focuses on a multi-dimensional  backward stochastic differential equation (BSDE)
	\begin{equation}\label{eq1}
	Y_{t} = \xi + \int_{t}^{T}H(s,\,Y_{u},\,Z_{u}) \ du - \int_{t}^{T}Z_{u} \ dW_{u}, \quad t \in [0,\, T],
	\end{equation}
where $W = (W^{1},\,\dots,\,W^{d})$ denotes a standard $ d $-dimensional Brownian motion defined on a filtered probability space
$(\Omega,\, \mathcal{F},\, (\mathcal{F}_{s})_{s \in [0, \, T]}, \mathbb{P})$. Here, $(\mathcal{F}_{s})_{s \in [0, \, T]}$ is the $\mathbb{P}$-completion of the filtration generated by $W$. The terminal condition $ \xi $ is an $ \mathcal{F}_{T}$-measurable $ \mathbb{R}^{n} $-valued random variable and the generator $H: [0, T] \times \Omega \times \mathbb{R}^{n} \times \mathbb{R}^{n\times d} \mapsto \mathbb{R}^{n} $ is a progressively measurable process. When $ n = 1 $, nonlinear BSDEs are pioneered under the Lipschitz condition by \citep{Peng1990}.  
BSDEs with generators that exhibit quadratic growth in regards to the variable $ z $ have been extensively researched by articles \citep{BriandHu2006, BriandHu2008, Tevzadze2008, 2013Monotone, BriandRichou2017} and others since \citep{Kobylanski2000} first investigated the situation of  bounded terminal conditions. This paper characterizes a specific category of BSDEs with the form of generators being $f(y)|z|^{2}$, which first appeared in \citep{DuffieEpstein1992} to our knowledge. 
 This paper characterizes a specific category of BSDEs with the form of generators given by $g(y)\lvert z \rvert^{2}$, which to our knowledge first appeared in \citep{DuffieEpstein1992}. 
By employing It\^o-Krylov's formula as well as a ``domination method''derived from the existence result for reflected BSDEs obtained in \citep{Essaky2011}, 
\citep{Bahlali2017} demonstrated existence and uniqueness results when $g$ is globally integrable on $\mathbb{R}$. Another intriguing case is $f(y,\, z)=|z|^{2}/y$, meaning that the generator $ f $ is singular at $ y=0 $. In this instance, \citep{Bahlali2018} assumed that the generator $ f $ satisfies
	\begin{equation}
	0 \leq f(s,\,\omega,\,y,\, z) \leq a_{s}+b_{s} y+\gamma_{s} z+\frac{\delta}{2y}\lvert z \rvert^{2},\, (s,\,\omega,\,y,\, z) \in [0,\, T]\times \Omega \times (0,\, +\infty) \times  \mathbb{R}^{1 \times d},
	\end{equation}
for some processes $ a,\, b,\, \gamma $, and a constant $ \delta $. By utilizing the domination method, they established the existence of solutions in $ \mathcal{S}^{p} \times \mathcal{L}^{2} $. Furthermore, they proved a uniqueness result for bounded solutions using techniques from convex duality.  Subsequently, many researchers has extended this work, such as \citep{Yang2017, Zheng2017, Tian2023, Zheng2024}.
This paper extends the result in \citep{Bahlali2018} to multi-dimensional BSDEs with diagonally quadratic generators and unbounded terminal conditions following the method in \citep{Fan2023}.

\section{Notations and Existence Results}\label{sec2}

This section establishes the existence results. Firstly, we need to introduce the notations used in this paper. We say a process or random variable satisfies some property if this holds except on the empty subset. Thus we sometimes omit \as. $ \mathbb{E}_{t} \left[  \cdot \right] := \mathbb{E}\left[  \cdot \mid \mathcal{F}_{t} \right] $ denotes the conditional mathematical expectation with respect to $ \mathcal{F}_{t} $. The set composed of stopping times $ \tau $ satisfying $ 0\le \tau \le T $ is represented by  $ \mathcal{T}_{0,T}  $. If $ N $ is an adapted and c\`adl\`ag process, define $ N^{\ast} := \sup_{t \in [0,T]} \lvert N_{t} \rvert $. For a matrix $ z=(z_{ij}) $, $z^{i}$ denotes the $i$-th row of $z$, and $ z' $ denotes the transpose of $ z $. We denote its norm by $\lvert z \rvert =\sqrt{\sum_{ij}\lvert z_{ij} \rvert^{2}}.$ The sign function is defined to be $ sgn(y) = \mathds{1}_{ \{y>0\} } - \mathds{1}_{ \{y \le 0\} }$. Let $ \mathbf{0}_{n \times d} $ represent the $ n \times d $ zero matrix and $\mathbf{1}_{n}$ denote the $ n $-dimensional all-ones vector. We recall that class $ (D) $ comprises progressively measurable processes $ X $ satisfying that $\{ X_{\tau}\mid \tau \in \mathcal{T}_{0,\,T} \}$ is uniformly integrable. Sometimes we denote by $ \mathscr{E}(X) $ the stochastic exponential of a one-dimensional local martingale $ X $ and denote by $ Z \cdot W $ the process $ (\int_{0}^{t} Z_{s} dW_{s})_{t \in [0,T]}.$ Given $p\geq 1$ and open sets $ U \subset \mathbb{R},\,V \subset \mathbb{R}^{n} $, let us define the following spaces and notations.

$\mathcal{C}^{p}(U)$: the space of functions from $U $ to $ \mathbb{R}$ having continuous p-th derivative.

$\mathcal{S}:=\mathcal{S}(\mathbb{R}^{n}) $, the space of adapted and continuous processes valued in $ \mathbb{R} $.

$ \mathbb{L}^{p}:= \mathbb{L}^{p}(\Omega, \mathcal{F}_{T}, \mathbb{P}; \mathbb{R}) $, the space of $ \mathcal{F}_{T}$-measurable random variables $\eta$ valued in $ \mathbb{R}$ such that $\mathbb{E}\left[\lvert \eta\rvert ^{p}\right] < \infty.$

$\mathcal{S}^{\infty}:= \mathcal{S}^{\infty}(\mathbb{R}^{n}) $, the space of all $ Y \in \mathcal{S}$ such that $\lVert Y \rVert_{\mathcal{S}^{\infty}} := \esssup\limits_{(\omega,t)}\lvert Y_{t}(\omega) \rvert < \infty.  $

$\mathcal{S}^{p}:= \mathcal{S}^{p}(\mathbb{R}^{n}) $, the space of all processes $Y\in \mathcal{S}$ such that $\lVert Y \rVert_{\mathcal{S}^{p}}:=\left(  \mathbb{E}[\sup\limits_{0 \le t \le T}\lvert Y_{t}\rvert ^{p}]\right)^{\frac{1}{p}} < \infty $. 

$\mathcal{L}^{2}:=\mathcal{L}^{2}(\mathbb{R}^{n\times d})$, the space of $\mathbb{R}^{n\times d}$-valued processes $N$ satisfying $N$ is progressively measurable and a.s. $ \int^{T}_{0}|N_{s}|^{2}ds < \infty. $

$\mathcal{M}^{p}:= \mathcal{M}^{p}(\mathbb{R}^{n\times d})$, the space of progressively measurable processes $G$ fulfilling that $G$ is valued in $\mathbb{R}^{n\times d}$ and $\lVert G \rVert_{\mathcal{M}^{p}}:=\left(  \mathbb{E}\left[\left(\int^{T}_{0}\lvert G_{s}\rvert ^{2}ds\right)^{p/2}\right]\right)^{1/p} < + \infty $. $Z \in \mathcal{M}(\mathbb{R}^{n\times d}) $ means that $Z \in \bigcap\limits_{p\ge1} \mathcal{M}^{p}(\mathbb{R}^{n\times d}).$

$BMO(\mathbb{R}^{n\times d})$: the space of all $ N \in \mathcal{M}^{2}$ satisfying \[ \lVert N \rVert_{BMO}: = \sup\limits_{\tau \in \mathcal{T}_{0,T }}\big\lVert \mathbb{E}\Big[ \int_{\tau}^{T}\lvert  N_{s}\rvert^{2} ds   \big| \mathcal{F}_{\tau} \Big]\big\rVert^{\frac{1}{2}}_{\infty} \! < \infty .\]
We note that for $ N \in BMO(\mathbb{R}^{n\times d}) $, the process $ N\cdot W $ is an $ n $-dimensional $ BMO $ martingale.

We sometimes describe the BSDE with generator $ H $ and terminal $ \xi $ as BSDE$ (\xi, H) $ rather than \myeqref{eq1} for notational simplicity.
A solution of BSDE$ (\xi, H) $ is defined as a pair of processes $(Y,\,Z):=\{(Y_{t},\,Z_{t})\}_{t \in [0,T]}  \in \mathcal{S} \times \mathcal{L}^{2} $ fulfilling that \as, \myeqref{eq1} holds and $ \int_{0}^{T} \lvert  g(u,\,Y_{u},\,Z_{u})\rvert  du < \infty $. Furthermore, we designate $ (Y,\,Z) := \{(Y_{t},\,Z_{t})\}_{t \in [0,T]}$ as a $ L^{p} $ solution provided that $ (Y,\,Z) \in \mathcal{S}^{p}\times  \mathcal{M}^{q}$ for some $ p>1,\,q>1 $. A solution $ (Y,\,Z) $ to BSDE$ (\xi, H) $  is said to be positive if each component of $ Y= (Y^{1},\,\dots,\,Y^{n}) $ satisfies $ Y^{i} > 0,\,i = 1\,\dots,\,n. $

Now we establish the existence theorem for BSDE \myeqref{eq1} under the following assumption:

\noindent {\textbf (H1)} There exist nonnegative and progressive measurable processes $\alpha,\,\beta$, nonnegative constants $A,\,\delta $ with $\delta \neq 1$, and scalar generators
$H^{i}_{1} : [0,\,T] \times \Omega \times \mathbb{R}^{n} \to \mathbb{R}$, and $H^{i}_{2} : [0,\,T] \times \Omega \times \mathbb{R} \times \mathbb{R}^{1\times d} \to \mathbb{R}$, $i =1,\,2,\,\dots,\,n,$ such that 
\begin{itemize}
	\item 
$ H=(H^{1},\,H^{2},\,\dots,\,H^{n}).$ For any $(s,\,\omega,\,y,\,z) \in [0,\,T] \times \Omega \times \mathbb{R}^{n} \times \mathbb{R}^{1\times d}$, and $i =1,\,2,\,\dots,\,n$, $H^{i}(s,\,\omega,\,y,\,z) = H_{1}^{i}(s,\,\omega,\,y) + H_{2}^{i}(s,\,\omega,\,y^{i},\,z^{i})$;
	\item for each $i,\,H^{i}_{1} $ is nonnegative, of linear growth in $y$, and uniformly Lipschitz continuous in $y$,  that is, \as, for all $s \in [0,\,T],\,y,\,y^{1},\,y^{2} \in \mathbb{R}^{n} $,
	\[0 \le H_{1}^{i}(s,\,y)  \le A \left( 1 + \lvert y\rvert \right),\text{ and }\lvert H_{1}^{i}(s,\,y^{1}) - H_{1}^{i}(s,\,y^{2}) \vert \le A \lvert y^{1} - y^{2} \rvert;\]
	\item \as, for each $ i $ and any $ s\in [0,\,T], $ $(y,\,z) \to H_{2}^{i}(s,\,\omega,\,y,\,z) $ is continuous and convex on $ \mathbb{R} \times \mathbb{R}^{1\times d}$;
	\item \as, for each $ i $ and any $(s,\,y,\,z) \in [0,\,T] \times (0,\,+\infty) \times \mathbb{R}^{1\times d}$,	\[0 \le  H_{2}^{i}(s,\,\omega,\,y,\,z) \le \alpha_{s} + \beta_{s} y + \frac{\delta}{2y}\lvert z \rvert^{2}.\]
\end{itemize}

\noindent {\textbf(H2)} $ \xi = (\xi^{1},\,\xi^{2},\,\dots,\,\xi^{n})$ satisfies that $ \xi^{i} > 0,\,i =1,\,2,\,\dots,\,n $. For some constants $ p \ge 2 $ and $ q > 1 $, \[\mathbb{E}\left[ \left(\frac{1}{ \xi^{i} } \right)^{q}\right]<+\infty,\,i =1,\,2,\,\dots,\,n,\,\mathbb{E}\left[ \left(1+\lvert \xi \rvert^{2+\delta}+\int_{0}^{T} \alpha^{2+\delta}_{s} ds\right)^{p}e^{p(2+\delta)\int_{0}^{T}\beta_{s}  d s } \right]<+\infty. \]

\begin{theorem}\label{mainresult}
	Assume (H1) and (H2) hold. Then BSDE \myeqref{eq1} admits a unique solution $ (Y,\,Z) $ such that $ (Y,\,Z) \in \mathcal{S}^{p(2+\delta)}(\mathbb{R}^{n}) \times \mathcal{M}^{p}(\mathbb{R}^{n\times d})$. Moreover, $ (Y,\,Z) $ is a positive solution.	
\end{theorem}

\begin{Corollary}\label{corollary}
Assume (H1) holds with processes $ \alpha $ and $ \beta $ are bounded. Assume also that $ \xi = (\xi^{1},\,\xi^{2},\,\dots,\,\xi^{n})$ is bounded and $ \xi^{i} > 0,\,i =1,\,2,\,\dots,\,n.$ For some constant $ q > 1 $, $\mathbb{E}\left[ \left(\frac{1}{ \xi^{i} } \right)^{q}\right]<\infty.$ Then BSDE \myeqref{eq1} admits a unique solution $ (Y,\,Z) $ such that $ Y\in \mathcal{S}^{\infty}(\mathbb{R}^{n})$ and $Z \in BMO(\mathbb{R}^{n \times d}).$
\end{Corollary}

The corollary follows immediately by applying It\^o's formula to $ e^{Mt}\lvert Y_{t}\rvert^{2} $ for some appropriate constant $ M $, so we omit its proof. To prove \sref{Theorem}{mainresult}, we introduce the follwing lemma, and more general cases can be found in  \citep{Yang2017}.
	\begin{lemma} \label{lemma1}
Consider the following one dimensional BSDE
		\begin{equation} \label{eq2}
		Y_{t}=\zeta  + \int_{t}^{T} g(s,\,\omega,\,	Y_{s},\,Z_{s})    d s-\int_{t}^{T} Z_{s} d W_{s},\,1 \le t \le T.
		\end{equation}
		Assume that $ \zeta >0 $, $ g $ is continuous in $ (y, z) $ for each fixed $ (s,\,\omega),$ and there exist nonnegative and progressively measurable processes $ a,\,b$, and a nonnegative constant $ \delta$, such that for any $(s,\,\omega,\,y,\,z) \in [0,\,T] \times \Omega \times (0,\,+\infty) \times \mathbb{R}^{1\times d}$,
	\[0 \le g(s,\,\omega,\,y,\,z) \le a_{s} + b_{s} y  + \frac{\delta}{2y}\lvert z \rvert^{2}.\]

	\begin{enumerate}
	\item If $ (Y,\,Z) $ is a positive solution of \myeqref{eq2} such that \[ \mathbb{E}\left[\left( \sup\limits_{0 \le t \le T}  Y_{t}^{2+\delta}  + \int_{0}^{T}a^{2+\delta}_{r}\ d r\right) e^{(2+\delta)\int_{0}^{T} b_{u} d u}  \right]< +\infty,\] then
	\[ e^{(1+\delta)t+(2+\delta)\int_{0}^{t} b_{u} d u}Y_{t}^{2+\delta} \le  \mathbb{E}\left[ e^{(1+\delta)T+(2+\delta)\int_{0}^{T} b_{u} d u}\zeta^{2+\delta} + \int_{t}^{T} e^{(1+\delta)s+(2+\delta)\int_{0}^{s} b_{u} d u}a^{2+\delta}_{s} \ d s\mid \mathcal{F}_{t}\right]. \]
	\item If $\delta \neq 1,\, (Y,\,Z) $ is a positive solution of \myeqref{eq2} such that \[\mathbb{E}\left[ \sup\limits_{0 \le s \le T}  Y_{s} ^{p(2+\delta)} +  \left(\int_{0}^{T}a_{s}  d s \right)^{p} +  \left(\int_{0}^{T}b_{s}  d s \right)^{p} \right] <+\infty \text{ for some }p > 1, \]  then
	\[
	\mathbb{E}\left[\left(\int_{0}^{T} \lvert Z_{u}\rvert ^{2} d u\right)^{\tfrac{p}{2}}\right] \leq  C+C\mathbb{E}\left[ \sup_{0 \le s \le T}  Y_{s} ^{p(2+\delta)} +  \left(\int_{0}^{T}a_{s}  d s \right)^{p}+ \left(\int_{0}^{T}b_{s}  d s \right)^{p} \right],
	\]	 where $ C $ is a constant depends on $ p,\,\delta,$ and $T. $
	\item If $\delta \neq 1,\,g(s,\,\omega,\,\cdot,\,\cdot) $ is convex in $(y,\,z)$ for each $s,$ \as, and for some $ p > 1 $, 
	\[\mathbb{E}\left[ \left(1 + \zeta^{2+\delta} + \int_{0}^{T} a^{2+\delta}_{r}\ d r\right)^{p}e^{p(2+\delta)\int_{0}^{T} b_{s}d s}\right] < +\infty,\]
	then the BSDE \myeqref{eq2} admits a unique solution $\left(Y,\, Z\right) \in \mathcal{S}^{p(2+\delta)}(\mathbb{R}) \times \mathcal{M}^{p}(\mathbb{R}^{1\times d}).$ Moreover,
	\[\mathbb{E}\left[ \sup\limits_{0 \le t \le T}  Y_{t} ^{p(2+\delta)}e^{p(2+\delta)\int_{0}^{t} b_{s}d s}\right] \le C \mathbb{E}\left[ \left(\zeta^{2+\delta} + \int_{0}^{T} a^{2+\delta}_{r}\ d r\right)^{p}e^{p(2+\delta)\int_{0}^{T} b_{s}d s}\right], \]
	where $ C > 0 $ is a contant depending on $ \delta,\,T,$ and $ p. $
\end{enumerate}
	\end{lemma}
\begin{proof}[Proof of \sref{Theorem}{mainresult}.]
	
	For any given $ (V,\,N) \in \mathcal{S}^{p(2+\delta)}(\mathbb{R}^{n}) \times \mathcal{M}^{p}(\mathbb{R}^{n\times d}) $ such that \[\mathbb{E}\left[ \left( \int_{0}^{T} \lvert V_{r} \rvert^{2+\delta}\ d r\right)^{p}e^{p(2+\delta)\int_{0}^{T} \beta_{s}d s}\right] < +\infty, \]  consider the following system of BSDEs:
	\begin{equation} \label{eq3}
y_{t}^{i} = \xi^{i} + \int_{t}^{T} H_{1}^{i}(r,\,V_{r}) + H_{2}^{i}(r,\,y^{i}_{r},\,z^{i}_{r})\ d r - \int_{t}^{T} z_{r}^{i}\ d W_{r},\,0 \le t \le T.
 	\end{equation}
	For each $ i =1,\,2,\,\dots,\,n, $ it is clear from (H1) that for any $ (r,\,\omega,\,y,\,z)\in [0,\,T] \times \Omega \times (0,\,+\infty) \times \mathbb{R}^{1\times d},$
\[ 0\le  g(r,\,\omega,\,y,\,z) := H_{1}^{i}(r,\,V_{r}) + H_{2}^{i}(r,\,\omega,\,y,\,z) \le A\left(1 + \lvert V_{r} \rvert\right) + \alpha_{r} +\beta_{r} y +\frac{\delta}{2y}\lvert z \rvert^{2}, \] and $ g $ is convex in $ (y,\,z) $.
	In view of (H2), we have that 
	\[\mathbb{E}\left[ \left(1+(\xi^{i})^{2+\delta} + \int_{0}^{T} \left( \alpha_{r} + A + A\lvert V_{r} \rvert\right) ^{2+\delta}\ d r\right)^{p}e^{p(2+\delta)\int_{0}^{T} \beta_{s}d s}\right] < +\infty,\]
	thus \sref{Lemma}{lemma1} shows that the system \myeqref{eq3} admits a unique solution $ (y,\,z) \in \mathcal{S}^{p(2+\delta)}(\mathbb{R}^{n}) \times \mathcal{M}^{p}(\mathbb{R}^{n\times d}). $ Moreover, for each $ i =1,\,2,\,\dots,\,n, $
		\begin{align*}
		&e^{(1+\delta)t+(2+\delta)\int_{0}^{t}\beta_{s}ds}(y^{i}_{t})^{2+\delta} \nonumber \\ \le & \mathbb{E}_{t}\left[ e^{(1+\delta)T+(2+\delta)\int_{0}^{T}\beta_{s}ds}\lvert \xi \rvert^{2+\delta} + \int_{t}^{T} e^{(1+\delta)s+(2+\delta)\int_{0}^{s}\beta_{u}du}\left(\alpha_{s}+ A + A\lvert V_{s} \rvert \right)^{2+\delta}d s\right]. 	\end{align*}
	Doob's $ L^{p} $ inequality gives that for any $t \in [0,\,T], $
	\begin{align}\label{eq4}
	&\mathbb{E}\left[\sup_{t \le s \le T}e^{(1+\delta)sp+(2+\delta)p\int_{0}^{s}\beta_{u}du}(y^{i}_{s})^{(2+\delta)p}  \right] \nonumber \\ 
\le &\left(\frac{p}{p-1}\right)^{p}3^{p-1}\mathbb{E}\left[ e^{(1+\delta)pT+(2+\delta)p\int_{0}^{T}\beta_{s}ds}\left( \lvert \xi \rvert^{(2+\delta)p} + 2^{(1+\delta)p}\left( \int_{t}^{T} \left( \alpha_{s}+ A\right)^{2+\delta} \ d s\right)^{p}\right)\right]\nonumber \\
+ & \left(\frac{p}{p-1}\right)^{p}3^{p-1}2^{(1+\delta)p}(T-t)^{p}A^{(2+\delta)p}\mathbb{E}\left[ \sup\limits_{t \le s \le T} \left(  e^{(1+\delta)sp+(2+\delta)p\int_{0}^{s}\beta_{u}du} \lvert V_{s} \rvert ^{(2+\delta)p}\right)\right].  	
	\end{align}

	Based on the above argument, we initialize $ (y^{0},\,z^{0}) = (\mathbf{1}_{n},\,\mathbf{0}_{n\times d}) $ and recursively define the sequence of processes $\{(y^{(m+1)},\,z^{(m+1)})\}_{m=0}^{\infty} \in \mathcal{S}^{p(2+\delta)}(\mathbb{R}^{n}) \times \mathcal{M}^{p}(\mathbb{R}^{n\times d})$ as the
	unique adapted solution to the system of BSDEs:
	\begin{equation} \label{eq5}
	y_{t}^{(m+1);i} = \xi^{i} + \int_{t}^{T} H_{1}^{i}(r,\,y_{r}^{(m)}) + H_{2}^{i}(r,\,y_{r}^{(m+1);i},\,z^{(m+1);i}_{r})\ d r - \int_{t}^{T} z_{r}^{(m+1);i}\ d W_{r},\,0 \le t \le T,
	\end{equation}
	where $i = 1,\,\dots,\, n,\,	y^{(m+1);i}   $ denotes the $ i $-th component of $ y^{(m+1)} $.
	Now we prove $\{(y^{(m+1)},\,z^{(m+1)})\}_{m=0}^{\infty}  $ is a Cauchy sequence in $ \mathcal{S}^{p(2+\delta)} \times \mathcal{M}^{p}$. 
	In view of \myeqref{eq4}, we have that for any $t \in [0,\,T], $
	\begin{align} \label{eq6}
	&\mathbb{E}\left[\sup_{t \le s \le T}e^{(1+\delta)sp+(2+\delta)p\int_{0}^{s}\beta_{u}du}\lvert y^{(m+1)} _{s} \rvert^{(2+\delta)p} \right] \nonumber \\ 
	\le &\left(\frac{p}{p-1}\right)^{p}n^{\frac{(2+\delta)p}{2}}3^{p-1}\mathbb{E}\left[ e^{(1+\delta)pT+(2+\delta)p \int_{0}^{T}\beta_{u}du}\left( \lvert \xi \rvert^{(2+\delta)p} + 2^{(1+\delta)p}\left( \int_{0}^{T} \left( \alpha_{s}+ A\right)^{2+\delta} \ d s\right)^{p}\right)\right]\nonumber \\
	+ & \left(\frac{p}{p-1}\right)^{p}n^{\frac{(2+\delta)p}{2}}3^{p-1}2^{(1+\delta)p}(T-t)^{p}A^{(2+\delta)p}\mathbb{E}\left[ \sup\limits_{t \le s \le T} \left(  e^{(1+\delta)sp+(2+\delta)p\int_{0}^{s}\beta_{u}du} \lvert y^{(m)}_{s} \rvert ^{(2+\delta)p}\right)\right]\nonumber\\
	 \le	&D(p) + B(t,p)\mathbb{E}\left[ \sup\limits_{t \le s \le T} \left(  e^{(1+\delta)sp+(2+\delta)p\beta s} \lvert y^{(m)}_{s} \rvert ^{(2+\delta)p}\right)\right], 	
		\end{align}
where \[ D(p) := \left(\tfrac{p}{p-1}\right)^{p}n^{\tfrac{(2+\delta)p}{2}}3^{p-1}\mathbb{E}\left[ e^{(1+\delta)pT+(2+\delta)p \int_{0}^{T}\beta_{u}du}\left( \lvert \xi \rvert^{(2+\delta)p} + 2^{(1+\delta)p}\left( \int_{0}^{T} \left( \alpha_{s}+ A\right)^{2+\delta} \ d s\right)^{p}\right)\right],\] $B(t,p):=\left(\frac{p}{p-1}\right)^{p}n^{\frac{(2+\delta)p}{2}}3^{p-1}2^{(1+\delta)p}(T-t)^{p}A^{(2+\delta)p} $.
If $ A = 0 , $	it is evident that for any $ t \in [0,\,T], $
\[\sup_{m \ge 0}\mathbb{E}\left[\sup_{t \le s \le T}e^{(1+\delta)sp+(2+\delta)p\int_{0}^{s}\beta_{u}du}\lvert y^{(m+1)} _{s} \rvert^{(2+\delta)p} \right] \le
D(p). \]
If 	$ A > 0 , $ set $ \varepsilon = \frac{1}{\frac{p}{p-1}3^{\frac{p-1}{p}}2^{1+\delta+\frac{1}{p}}	A^{2+\delta}n^{\frac{2+\delta}{2}} } $. Then $ \varepsilon > 0. $ Let $ m_{0} $ be the unique integer such that
\begin{equation}\label{eq8}
	\frac{T}{\varepsilon} \le  m_{0} < \frac{T}{\varepsilon}  +1 .
\end{equation} 
If $ m_{0} = 1, $ that is, $ 2B(t,p)\le 2\left(\frac{p}{p-1}\right)^{p}n^{\frac{(2+\delta)p}{2}}3^{p-1}2^{(1+\delta)p}A^{(2+\delta)p}T^{p} = \left(\frac{T}{\varepsilon} \right) ^{p}\le 1. $ Therefore, it follows from \myeqref{eq6} that for any $m \ge 0,\, t \in [0,\,T], $
\begin{align}\label{eq7}
   &\mathbb{E}\left[\sup_{t \le s \le T}e^{(1+\delta)sp+(2+\delta)p\int_{0}^{s}\beta_{u}du}\lvert y^{(m+1)} _{s} \rvert^{(2+\delta)p} \right]  \nonumber \\ \le
&D(p) + \frac{1}{2}\mathbb{E}\left[ \sup\limits_{t \le s \le T} \left(  e^{(1+\delta)sp+(2+\delta)p\int_{0}^{s}\beta_{u}du} \lvert y^{(m)}_{s} \rvert ^{(2+\delta)p}\right)\right] \nonumber\\
    \le
&	D(p)+\frac{1}{2}D(p)+\frac{1}{4}\mathbb{E}\left[ \sup\limits_{t \le s \le T} \left(  e^{(1+\delta)sp+(2+\delta)p\int_{0}^{s}\beta_{u}du} \lvert y^{(m-1)}_{s} \rvert ^{(2+\delta)p}\right)\right] \nonumber \\
    \le &
D(p)\left(1+\frac{1}{2}+...+\frac{1}{2^{m}}\right)+\frac{1}{2^{m+1}}\mathbb{E}\left[ \sup\limits_{t \le s \le T} \left(  e^{(1+\delta)sp+(2+\delta)p\int_{0}^{s}\beta_{u}du} \lvert y^{(0)}_{s} \rvert ^{(2+\delta)p}\right)\right]\nonumber \\
\le&  2D(p)+\frac{1}{2}\mathbb{E}\left[ \sup\limits_{t \le s \le T} \left(  e^{(1+\delta)sp+(2+\delta)p\int_{0}^{s}\beta_{u}du} \lvert y^{(0)}_{s} \rvert ^{(2+\delta)p}\right)\right].
\end{align} 
Therefore,
\begin{align*}
&\sup_{m \ge 1}\mathbb{E}\left[\sup_{0 \le s \le T}e^{(1+\delta)sp+(2+\delta)p\int_{0}^{s}\beta_{u}du}\lvert y^{(m)} _{s} \rvert^{(2+\delta)p}\right] \nonumber\\ \le &2D(p)+\frac{1}{2}\mathbb{E}\left[ \sup\limits_{0 \le s \le T} \left(  e^{(1+\delta)sp+(2+\delta)p\int_{0}^{s}\beta_{u}du} \lvert y^{(0)}_{s} \rvert ^{(2+\delta)p}\right)\right]< +\infty.\end{align*}

If $ m_{0} = 2$, then $ \varepsilon < T \le 2\varepsilon. $
Thus for all $ t \in \left[T-\varepsilon,\,T\right],\, 2B(t,p)  = \left(\frac{T-t}{\varepsilon} \right) ^{p}  \le 1. $ Similarly, we have that 
\begin{align}\label{eq9}
&\sup_{m \ge 1}\mathbb{E}\left[\sup_{T-\varepsilon \le s \le T}e^{(1+\delta)sp+(2+\delta)p\int_{0}^{s}\beta_{u}du}\lvert y^{(m)} _{s} \rvert^{(2+\delta)p} \right] \nonumber\\ \le &2D(p)+\frac{1}{2}\mathbb{E}\left[ \sup\limits_{0 \le s \le T} \left(  e^{(1+\delta)sp+(2+\delta)p\int_{0}^{s}\beta_{u}du} \lvert y^{(0)}_{s} \rvert ^{(2+\delta)p}\right)\right] < +\infty.\end{align} 
In particular,
\begin{align}\label{eq10}
&\sup_{m \ge 1} \mathbb{E}\left[e^{(1+\delta)(T-\varepsilon)p+(2+\delta)p\int_{0}^{T-\varepsilon}\beta_{u}du}\lvert y^{(m)} _{T-\varepsilon} \rvert^{(2+\delta)p} \right] \nonumber\\ \le &2D(p)+\frac{1}{2}\mathbb{E}\left[ \sup\limits_{0\le s \le T} \left(  e^{(1+\delta)sp+(2+\delta)p\int_{0}^{s}\beta_{u}du} \lvert y^{(0)}_{s} \rvert ^{(2+\delta)p}\right)\right] < +\infty. \end{align} 
Consider the following system of BSDEs for any $m \ge 0,\,0 \le t \le T-\varepsilon,$
\begin{equation} \label{eq11}
y_{t}^{(m+1);i} = y^{(m+1);i}_{T-\varepsilon} + \int_{t}^{T-\varepsilon} H_{1}^{i}(r,\,y_{r}^{(m)}) + H_{2}^{i}(r,\,y_{r}^{(m+1);i},\,z^{(m+1);i}_{r})\ d r - \int_{t}^{T-\varepsilon} z_{r}^{(m+1);i}\ d W_{r},
\end{equation}
where $i = 1,\,\dots,\, n$. Noting the fact that $ T-\varepsilon < \varepsilon $, we obtain an inequality similar to \myeqref{eq7} for any $ m \ge 0,\,t \in [0,\,T-\varepsilon], $
	\begin{align} \label{eq12}
&\mathbb{E}\left[\sup_{t \le s \le T-\varepsilon}e^{(1+\delta)sp+(2+\delta)p\int_{0}^{s}\beta_{u}du}\lvert y^{(m+1)} _{s} \rvert^{(2+\delta)p}\right]  \nonumber\\
\le	 &\left(\frac{p}{p-1}\right)^{p}n^{\frac{(2+\delta)p}{2}}3^{p-1}\mathbb{E}\left[ e^{(1+\delta)p(T-\varepsilon)+(2+\delta)p \int_{0}^{T-\varepsilon}\beta_{u}du}\lvert y^{(m+1)}_{T-\varepsilon} \rvert^{(2+\delta)p} \right]
\nonumber \\
+ & \left(\frac{p}{p-1}\right)^{p}n^{\frac{(2+\delta)p}{2}}3^{p-1}\mathbb{E}\left[ e^{(1+\delta)p(T-\varepsilon)+(2+\delta)p \int_{0}^{T-\varepsilon}\beta_{u}du}2^{(1+\delta)p}\left( \int_{0}^{T-\varepsilon} \left( \alpha_{s}+ A\right)^{2+\delta} \ d s\right)^{p}\right]
\nonumber \\
+ & \left(\frac{p}{p-1}\right)^{p}n^{\frac{(2+\delta)p}{2}}3^{p-1}2^{(1+\delta)p}(T-\varepsilon-t)^{p}A^{(2+\delta)p}\mathbb{E}\left[ \sup\limits_{t \le s \le T-\varepsilon} \left(  e^{(1+\delta)sp+(2+\delta)p\int_{0}^{s}\beta_{u}du} \lvert y^{(m)}_{s} \rvert ^{(2+\delta)p}\right)\right]\nonumber\\
\le	&
\left(\frac{p}{p-1}\right)^{p}n^{\frac{(2+\delta)p}{2}}3^{p-1}\left(2D(p)+\frac{1}{2}\mathbb{E}\left[ \sup\limits_{0\le s \le T} \left(  e^{(1+\delta)sp+(2+\delta)p\int_{0}^{s}\beta_{u}du} \lvert y^{(0)}_{s} \rvert ^{(2+\delta)p}\right)\right]\right)
\nonumber\\
+	&D(p)+\frac{1}{2}\mathbb{E}\left[ \sup\limits_{t \le s \le T-\varepsilon} \left(  e^{(1+\delta)sp+(2+\delta)p\int_{0}^{s}\beta_{u}du} \lvert y^{(m)}_{s} \rvert ^{(2+\delta)p}\right)\right]\nonumber\\
=	&
\bar{D}(p)+\frac{1}{2}\mathbb{E}\left[ \sup\limits_{t \le s \le T-\varepsilon} \left(  e^{(1+\delta)sp+(2+\delta)p\int_{0}^{s}\beta_{u}du} \lvert y^{(m)}_{s} \rvert ^{(2+\delta)p}\right)\right]\nonumber\\
\le &
\bar{D}(p)\left(1+\frac{1}{2}+...+\frac{1}{2^{m}}\right)+\frac{1}{2^{m+1}}\mathbb{E}\left[ \sup\limits_{t \le s \le T-\varepsilon} \left(  e^{(1+\delta)sp+(2+\delta)p\int_{0}^{s}\beta_{u}du} \lvert y^{(0)}_{s} \rvert ^{(2+\delta)p}\right)\right]\nonumber \\
\le&  2\bar{D}(p)+\frac{1}{2}\mathbb{E}\left[ \sup\limits_{0 \le s \le T} \left(  e^{(1+\delta)sp+(2+\delta)p\int_{0}^{s}\beta_{u}du} \lvert y^{(0)}_{s} \rvert ^{(2+\delta)p}\right)\right],
\end{align}
where \begin{align*}
	\bar{D}(p) := &\left( 1+2\left(\frac{p}{p-1}\right)^{p}n^{\frac{(2+\delta)p}{2}}3^{p-1}\right)D(p)\nonumber \\+&\left(\frac{p}{p-1}\right)^{p}n^{\frac{(2+\delta)p}{2}}3^{p-1}\frac{1}{2}\mathbb{E}\left[ \sup\limits_{0\le s \le T} \left(  e^{(1+\delta)sp+(2+\delta)p\int_{0}^{s}\beta_{u}du} \lvert y^{(0)}_{s} \rvert ^{(2+\delta)p}\right)\right].
\end{align*}  
Therefore,
	\begin{align} \label{eq13}
&\sup_{m \ge 1}\mathbb{E}\left[\sup_{0 \le s \le T-\varepsilon}e^{(1+\delta)sp+(2+\delta)p\int_{0}^{s}\beta_{u}du}\lvert y^{(m)} _{s} \rvert^{(2+\delta)p}\right]  \nonumber\\
\le	 &\left( 1+2\left(\frac{p}{p-1}\right)^{p}n^{\frac{(2+\delta)p}{2}}3^{p-1}\right)2D(p)\nonumber \\+&\left( 1+2\left(\frac{p}{p-1}\right)^{p}n^{\frac{(2+\delta)p}{2}}3^{p-1}\right)\frac{1}{2}\mathbb{E}\left[ \sup\limits_{0\le s \le T} \left(  e^{(1+\delta)sp+(2+\delta)p\int_{0}^{s}\beta_{u}du} \lvert y^{(0)}_{s} \rvert ^{(2+\delta)p}\right)\right].
\end{align}
Combining \myeqref{eq9} and \myeqref{eq13} gives that
\begin{align}\label{eq14}
&\sup_{m \ge 1}\mathbb{E}\left[\sup_{0 \le s \le T}e^{(1+\delta)sp+(2+\delta)p\int_{0}^{s}\beta_{u}du}\lvert y^{(m)} _{s} \rvert^{(2+\delta)p} \right] \nonumber\\
\le &\left( 1+2\left(\frac{p}{p-1}\right)^{p}n^{\frac{(2+\delta)p}{2}}3^{p-1}\right)2D(p)\nonumber \\+&\left( 1+2\left(\frac{p}{p-1}\right)^{p}n^{\frac{(2+\delta)p}{2}}3^{p-1}\right)\frac{1}{2}\mathbb{E}\left[ \sup\limits_{0\le s \le T} \left(  e^{(1+\delta)sp+(2+\delta)p\int_{0}^{s}\beta_{u}du} \lvert y^{(0)}_{s} \rvert ^{(2+\delta)p}\right)\right]<\infty.
\end{align} 
If $ m_{0} = 3, $ then a similar induction gives
\begin{align}\label{eq15}
&\sup_{m \ge 1}\mathbb{E}\left[\sup_{0 \le s \le T}e^{(1+\delta)sp+(2+\delta)p\int_{0}^{s}\beta_{u}du}\lvert y^{(m)} _{s} \rvert^{(2+\delta)p} \right] \nonumber\\
\le &\left( 1+2\left(\frac{p}{p-1}\right)^{p}n^{\frac{(2+\delta)p}{2}}3^{p-1}+\left(2\left(\frac{p}{p-1}\right)^{p}n^{\frac{(2+\delta)p}{2}}3^{p-1}\right)^{2}\right)\nonumber \\\cdot&\left(2D(p)+ \frac{1}{2}\mathbb{E}\left[ \sup\limits_{0\le s \le T} \left(  e^{(1+\delta)sp+(2+\delta)p\int_{0}^{s}\beta_{u}du} \lvert y^{(0)}_{s} \rvert ^{(2+\delta)p}\right)\right]\right)<\infty.
\end{align} 

Therefore, 
\begin{align}\label{eq16}
&\sup_{m \ge 1}\mathbb{E}\left[\sup_{0 \le s \le T}e^{(1+\delta)sp+(2+\delta)p\int_{0}^{s}\beta_{u}du}\lvert y^{(m)} _{s} \rvert^{(2+\delta)p} \right] \nonumber\\
\le &\left( 1+2\left(\frac{p}{p-1}\right)^{p}n^{\frac{(2+\delta)p}{2}}3^{p-1}+\dots+\left(2\left(\frac{p}{p-1}\right)^{p}n^{\frac{(2+\delta)p}{2}}3^{p-1}\right)^{m_{0}-1}\right)\nonumber \\\cdot&\left(2D(p)+ \frac{1}{2}\mathbb{E}\left[ \sup\limits_{0\le s \le T} \left(  e^{(1+\delta)sp+(2+\delta)p\int_{0}^{s}\beta_{u}du} \lvert y^{(0)}_{s} \rvert ^{(2+\delta)p}\right)\right]\right)<\infty.
\end{align} 

Now we prove that
\begin{align}\label{eq17}
	\sup_{m \ge 1}\mathbb{E}\left[\left( \int_{0}^{T} \lvert z^{(m)}_{r}\rvert^{2}\ d r\right)^{\frac{p}{2}} \right]<\infty.
\end{align}
	 	From \sref{Lemma}{lemma1}, we have for any $ m \ge 0, $
	 \[
	 \mathbb{E}\left[\left(\int_{0}^{T} \lvert z^{(m+1);i}_{u}\rvert ^{2} d u\right)^{\tfrac{p}{2}}\right] \leq  C\left\{1+ \sup_{m \ge 0}\mathbb{E}\left[ \sup_{0 \le s \le T}  \lvert y^{(m)}_{s} \rvert^{p(2+\delta)} \right]+ \mathbb{E}\left[ \left(\int_{0}^{T}\alpha_{s} d s \right)^{p}+\left(\int_{0}^{T}\beta_{s} d s \right)^{p} \right]\right\},
	 \]
	 where $ C $ is a contant independent on $ m. $ Consequently, in view of (H2), we have
	 \[\sup_{m \ge 0}\mathbb{E}\left[\left(\int_{0}^{T} \lvert z^{(m)}_{u}\rvert ^{2} d u\right)^{\tfrac{p}{2}}\right] <\infty.  \]
	Now we prove that $\{(y^{(m)},\,z^{(m)})\}_{m=1}^{\infty}  $ is a Cauchy sequence in $ \mathcal{S}^{p(2+\delta)} \times \mathcal{M}^{p} $.
	For any given $m,\,q \ge 1,\,\theta \in (0,\, 1) $, set
\[
\Delta_{\theta}y^{(m,q)}_{t}:=\frac{y^{(m+q)}_{t}-\theta y_{t}^{(m)}}{1-\theta}, \  \Delta_{\theta} z^{(m,q)}_{t} := \frac{z^{(m+q)}_{t}-\theta z^{(m)}_{t}}{1-\theta}.
\] 
	Therfore, $(\Delta_{\theta}y^{(m,q)},\,\Delta_{\theta} z^{(m,q)})  $ is a solution to the following system of BSDEs
	\begin{align*}
		\Delta_{\theta}y^{(m,q);i}_{t} = \xi^{i} &+ \int_{t}^{T} \frac{1}{1-\theta}\Big[H_{1}^{i}\left(u,\, y^{(m+q-1)}_{u}\right)-\theta H_{1}^{i}\left(u,\, y^{(m-1)}_{u}\right)\nonumber\\&+H_{2}^{i}\left(u,\, y_{u}^{(m+q);i},\, z_{u}^{(m+q);i}\right)-\theta H_{2}^{i}\left(u,\, y_{u}^{(m);i},\, z_{u}^{(m);i}\right)\Big]du - \int_{t}^{T} \Delta_{\theta} z^{(m,q);i}_{u}d W_{u}.
	\end{align*}	
It follows from (H1) that
\begin{align}\label{eq18}
	&\frac{\mathds{1}_{\Delta_{\theta}y^{(m,q);i}_{u}>0}}{1-\theta}\left[H_{1}^{i}\left(u,\, y^{(m+q-1)}_{u}\right)-\theta H_{1}^{i}\left(u,\, y^{(m-1)}_{u}\right)+H_{2}^{i}\left(u,\, y_{u}^{(m+q);i},\, z_{u}^{(m+q);i}\right)-\theta H_{2}^{i}\left(u,\, y_{u}^{(m);i},\, z_{u}^{(m);i}\right)\right] \nonumber \\
	\le &\frac{A}{1-\theta}\lvert y^{(m+q-1)}_{u}- y^{(m-1)}_{u}\rvert+H^{i}_{1}\left(u,\, y^{(m-1)}_{u}\right)+ \mathds{1}_{\Delta_{\theta}y^{(m,q);i}_{u}>0}H_{2}^{i}\left(u,\, 	\Delta_{\theta}y^{(m,q);i}_{u},\, 	\Delta_{\theta}z^{(m,q);i}_{u}\right)\nonumber \\
	\le &A\left[\lvert \Delta_{\theta}y^{(m-1,q)}_{u}\rvert +1+2\lvert  y^{(m-1)}_{u}\rvert\right]+\alpha_{u}+\beta_{u}(\Delta_{\theta}y^{(m,q);i}_{u})^{+}+\frac{\delta\mathds{1}_{\Delta_{\theta}y^{(m,q);i}_{u}>0}}{2\Delta_{\theta}y^{(m,q);i}_{u}}\lvert\Delta_{\theta}z^{(m,q);i}_{u}\rvert^{2}.
\end{align}	
	Applying Tanaka’s formula and It\^o's formula to $ ((\Delta_{\theta}y^{(m,q);i}_{t})^{+})^{2+\delta}e^{(1+\delta)t+(2+\delta)\int_{0}^{t}\beta_{u}du} $ and using \myeqref{eq16}, we have for $ t \in [0,\,T], $
\begin{align*}
	&((\Delta_{\theta}y^{(m,q);i}_{t})^{+})^{2+\delta}e^{(1+\delta)t+(2+\delta)\int_{0}^{t}\beta_{u}du} \nonumber \\ \le & \mathbb{E}_{t}\!\left[ \left(\xi^{i}\right)^{2+\delta}e^{(1+\delta)T+(2+\delta)\int_{0}^{T}\beta_{u}du}+\!\! \int_{t}^{T}\!e^{(1+\delta)s+(2+\delta)\int_{0}^{s}\beta_{u}du}\left(A+\alpha_{s}+2A\lvert  y^{(m-1)}_{s}\rvert+A\lvert \Delta_{\theta}y^{(m-1,q)}_{s}\rvert\right)^{2+\delta}d s\right]. 
\end{align*} 
	Analogously, define
	\[
	\Delta_{\theta}\tilde{y}^{(m,q)}_{t}:=\frac{y^{(m)}_{t}-\theta y_{t}^{(m+q)
	}}{1-\theta}, \  \Delta_{\theta} \tilde{z}^{(m,q)}_{t} := \frac{z^{(m)}_{t}-\theta z^{(m+q)}_{t}}{1-\theta}.
	\] 
	Then we have
\begin{align*}
&((\Delta_{\theta}\tilde{y}^{(m,q);i}_{t})^{+})^{2+\delta}e^{(1+\delta)t+(2+\delta)\int_{0}^{t}\beta_{u}du} \nonumber \\ \le & \mathbb{E}_{t}\!\left[ \left(\xi^{i}\right)^{2+\delta}e^{(1+\delta)T+(2+\delta)\int_{0}^{T}\beta_{u}du}+\!\! \int_{t}^{T}\!e^{(1+\delta)s+(2+\delta)\int_{0}^{s}\beta_{u}du}\left(A+\alpha_{s}+2A\lvert  y^{(m+q-1)}_{s}\rvert+A\lvert \Delta_{\theta}\tilde{y}^{(m-1,q)}_{s}\rvert\right)^{2+\delta}d s\right]. 
\end{align*} 
	Note the facts that
	\[\left(\Delta_{\theta}y^{(m,q);i}_{t} \right)^{-}  = \frac{\left(\theta  y_{t}^{(m);i} - y^{(m+q);i}_{t}\right)^{+}}{1-\theta} \le \left(\Delta_{\theta}\tilde{y}^{(m,q);i}_{t} \right)^{+},\]
	and \[\left(\Delta_{\theta}\tilde{y}^{(m,q);i}_{t} \right)^{-}  \le \left(\Delta_{\theta}y^{(m,q);i}_{t} \right)^{+}. \]
	Thus,
	\[\lvert\Delta_{\theta}y^{(m,q)}_{t} \rvert^{2+\delta},\,\lvert\Delta_{\theta}\tilde{y}^{(m,q)}_{t} \rvert^{2+\delta} \nonumber \le n^{\frac{2+\delta}{2}}2^{1+\delta}\left(\max_{i}((\Delta_{\theta}y^{(m,q);i}_{t} )^{+})^{2+\delta}+\max_{i}((\Delta_{\theta}\tilde{y}^{(m,q);i}_{t} )^{+})^{2+\delta}\right).\]
	Consequently, we have
	\begin{align}\label{eq19}
		&\left( \lvert\Delta_{\theta}y^{(m,q)}_{t} \rvert ^{2+\delta}+ \lvert\Delta_{\theta}\tilde{y}^{(m,q)}_{t} \rvert ^{2+\delta} \right)e^{(1+\delta)t+(2+\delta)\int_{0}^{t}\beta_{u}du} \nonumber \\ \le &
			n^{\frac{2+\delta}{2}}2^{3+\delta}\mathbb{E}_{t}\left[ \lvert\xi\rvert^{2+\delta}e^{(1+\delta)T+(2+\delta)\int_{0}^{T}\beta_{u}du}\right]\nonumber \\ + &
			n^{\frac{2+\delta}{2}}2^{5+3\delta}\mathbb{E}_{t}\left[\int_{t}^{T}e^{(1+\delta)s+(2+\delta)\int_{0}^{s}\beta_{u}du}\left(A+\alpha_{s}\right)^{2+\delta}d s\right]\nonumber \\ + &
			n^{\frac{2+\delta}{2}}2^{6+4\delta}A^{2+\delta}\mathbb{E}_{t}\left[\int_{t}^{T}e^{(1+\delta)s+(2+\delta)\int_{0}^{s}\beta_{u}du}\left(\lvert  y^{(m-1)}_{s}\rvert^{2+\delta}+\lvert  y^{(m+q-1)}_{s}\rvert^{2+\delta}\right)d s\right]\nonumber \\ + &
			n^{\frac{2+\delta}{2}}2^{3+2\delta}A^{2+\delta}\mathbb{E}_{t}\left[\int_{t}^{T}e^{(1+\delta)s+(2+\delta)\int_{0}^{s}\beta_{u}du}\left(\lvert \Delta_{\theta}\tilde{y}^{(m-1,q)}_{s}\rvert^{2+\delta}+\lvert \Delta_{\theta}y^{(m-1,q)}_{s}\rvert^{2+\delta}\right)d s\right].
	\end{align}
 
	Therefore, Doob's $ L^{p} $ inequality gives for any $ t \in [0,\,T], $
	\begin{align}\label{eq20}
	&\mathbb{E}\left[\sup_{t \le s \le T}  \left( \lvert\Delta_{\theta}y^{(m,q)}_{s} \rvert ^{2+\delta}+ \lvert\Delta_{\theta}\tilde{y}^{(m,q)}_{s} \rvert ^{2+\delta} \right)^{p}e^{(1+\delta)sp+(2+\delta)p\int_{0}^{s}\beta_{u}du}  \right]\nonumber \\ \le & (\frac{p}{p-1})^{p}4^{p-1}	n^{\frac{(2+\delta)p}{2}}2^{(5+3\delta)p}\mathbb{E}\left[\left( \lvert\xi\rvert^{(2+\delta)p}+\left( \int_{0}^{T}\left(A+\alpha_{s}\right)^{2+\delta}d s\right)^{p}\right)e^{(1+\delta)pT+(2+\delta)p\int_{0}^{T}\beta_{u}du}\right]\nonumber \\ + &
	(\frac{p}{p-1})^{p}4^{p-1}n^{\frac{(2+\delta)p}{2}}2^{(7+4\delta)p}A^{(2+\delta)p}T^{p}\sup_{m \ge 1}\mathbb{E}\left[\sup_{0 \le s \le T}e^{(1+\delta)sp+(2+\delta)p\int_{0}^{s}\beta_{u}du}\lvert y^{(m)} _{s} \rvert^{(2+\delta)p} \right]\nonumber \\ + &
	(\frac{p}{p-1})^{p}4^{p-1}n^{\frac{(2+\delta)p}{2}}2^{(3+2\delta)p}A^{(2+\delta)p}(T-t)^{p}\nonumber \\ \cdot &\mathbb{E}\left[\sup_{t \le s \le T}e^{(1+\delta)ps+(2+\delta)p\int_{0}^{s}\beta_{u}du}\left(\lvert \Delta_{\theta}\tilde{y}^{(m-1,q)}_{s}\rvert^{2+\delta}+\lvert \Delta_{\theta}y^{(m-1,q)}_{s}\rvert^{2+\delta}\right)^{p}\right]\nonumber \\ = &
	\tilde D(p) + \tilde B(t,\,p)\mathbb{E}\left[\sup_{t \le s \le T}e^{(1+\delta)ps+(2+\delta)p\int_{0}^{s}\beta_{u}du}\left(\lvert \Delta_{\theta}\tilde{y}^{(m-1,q)}_{s}\rvert^{2+\delta}+\lvert \Delta_{\theta}y^{(m-1,q)}_{s}\rvert^{2+\delta}\right)^{p}\right],
	\end{align}
	where $ \tilde D(p) $ is a constant independent on $m$ and $q,\,  \tilde B(t,\,p) =	(\frac{p}{p-1})^{p}4^{p-1}n^{\frac{(2+\delta)p}{2}}2^{(3+2\delta)p}A^{(2+\delta)p}(T-t)^{p}.$
If $ A = 0 , $	it is evident that for any $ t \in [0,\,T], $
\[\mathbb{E}\left[\sup_{t \le s \le T}\left( \lvert\Delta_{\theta}y^{(m,q)}_{s} \rvert ^{2+\delta}+ \lvert\Delta_{\theta}\tilde{y}^{(m,q)}_{s} \rvert ^{2+\delta} \right)^{p}e^{(1+\delta)sp+(2+\delta)p\int_{0}^{s}\beta_{u}du}  \right] \le
\tilde D(p). \]
If 	$ A > 0 , $ set $ \tilde \varepsilon = \frac{1}{\frac{p}{p-1}4^{\frac{p-1}{p}}2^{3+2\delta+\frac{1}{p}}	A^{2+\delta}n^{\frac{2+\delta}{2}} } $. Then $\tilde \varepsilon > 0. $ Let $ \tilde m_{0} $ be the unique integer such that
\begin{equation}\label{eq21}
\frac{T}{\tilde \varepsilon} \le \tilde m_{0} < \frac{T}{\tilde \varepsilon}  +1 .
\end{equation} 
 Following a similar process as used to derive \myeqref{eq16}, we can obtain that for any $m,\,q \ge 1 $, 
		\begin{align}\label{eq22}
		&\mathbb{E}\left[\sup_{0 \le s \le T}\left( \lvert\Delta_{\theta}y^{(m,q)}_{s} \rvert ^{2+\delta}+ \lvert\Delta_{\theta}\tilde{y}^{(m,q)}_{s} \rvert ^{2+\delta} \right)^{p}e^{(1+\delta)sp+(2+\delta)p\int_{0}^{s}\beta_{u}du} \right] \nonumber\\
		\le &\left( 1+2^{(2+\delta)p+1}\left(\frac{p}{p-1}\right)^{p}n^{\frac{(2+\delta)p}{2}}4^{p-1}+\dots+\left(2^{(2+\delta)p+1}\left(\frac{p}{p-1}\right)^{p}n^{\frac{(2+\delta)p}{2}}4^{p-1}\right)^{\tilde m_{0}-1}\right)\nonumber \\\cdot&\left\{2\tilde D(p)+ \frac{1}{2^{m}}\mathbb{E}\left[ \sup\limits_{0\le s \le T}   e^{(1+\delta)sp+(2+\delta)p\int_{0}^{s}\beta_{u}du} \left(\lvert \Delta_{\theta}y^{(0,q)}_{s}  \rvert ^{2+\delta}+\lvert \Delta_{\theta}\tilde y^{(0,q)}_{s}  \rvert ^{2+\delta}\right)^{p}\right]\right\}\nonumber\\
		= & C(p)\left\{2\tilde D(p)+ \frac{1}{2^{m}}\mathbb{E}\left[ \sup\limits_{0\le s \le T}   e^{(1+\delta)sp+(2+\delta)p\int_{0}^{s}\beta_{u}du} \left(\lvert \Delta_{\theta}y^{(0,q)}_{s}  \rvert ^{2+\delta}+\lvert \Delta_{\theta}\tilde y^{(0,q)}_{s}  \rvert ^{2+\delta}\right)^{p}\right]\right\},
	\end{align}	
	where $ C(p) = 1+2^{(2+\delta)p+1}\left(\frac{p}{p-1}\right)^{p}n^{\frac{(2+\delta)p}{2}}4^{p-1}+\dots+\left(2^{(2+\delta)p+1}\left(\frac{p}{p-1}\right)^{p}n^{\frac{(2+\delta)p}{2}}4^{p-1}\right)^{\tilde m_{0}-1}.$
	In view of $ \lvert y^{(m+q)}_{s}-y^{(m)}_{s} \rvert^{(2+\delta)p} \le 2^{(2+\delta)p-1}\left(\lvert y^{(m+q)}_{s}-\theta y^{(m)}_{s} \rvert^{(2+\delta)p}+(1-\theta)^{(2+\delta)p}\lvert y^{(m)}_{s} \rvert^{(2+\delta)p}\right) $, we have
		\begin{align}\label{eq23}
&	\sup\limits_{q \ge 1}\mathbb{E}\left[\sup\limits_{0 \le s \le T} \lvert y^{(m+q)}_{s}-y^{(m)}_{s} \rvert^{(2+\delta)p}e^{(2+\delta)p\int_{0}^{s}\beta_{u}du}\right] \nonumber\\
 \le & 2^{(2+\delta)p-1}(1-\theta)^{(2+\delta)p}\left\{\mathbb{E}\left[\sup\limits_{0 \le s \le T}\lvert y^{(m)}_{s} \rvert^{(2+\delta)p}e^{(2+\delta)p\int_{0}^{s}\beta_{u}du}\right]+ \sup\limits_{q \ge 1}\mathbb{E}\left[\sup\limits_{0 \le s \le T}\lvert \Delta_{\theta}y^{(m,q)}_{s} \rvert^{(2+\delta)p}e^{(2+\delta)p\int_{0}^{s}\beta_{u}du}\right]\right\}\nonumber\\
 \le & 2^{(2+\delta)p-1}(1-\theta)^{(2+\delta)p}\left\{C(p)2\tilde D(p)+\mathbb{E}\left[\sup\limits_{0 \le s \le T}\lvert y^{(m)}_{s} \rvert^{(2+\delta)p}e^{(2+\delta)p\int_{0}^{s}\beta_{u}du}\right]\right\}\nonumber\\+ &2^{(2+\delta)p-1}\frac{C(p)}{2^{m}}\sup\limits_{q \ge 1}\mathbb{E}\left[ \sup\limits_{0\le s \le T}   e^{(1+\delta)sp+(2+\delta)p\int_{0}^{s}\beta_{u}du} \left(\lvert y^{(q)}_{s}-\theta y^{(0)}_{s} \rvert ^{2+\delta}+\lvert y^{(0)}_{s}-\theta y^{(q)}_{s}  \rvert ^{2+\delta}\right)^{p}\right].
 \end{align}	
	Therefore, 
		\begin{align}\label{eq24}
&	\lim\limits_{m \to +\infty}	\sup\limits_{q \ge 1}\mathbb{E}\left[\sup\limits_{0 \le s \le T} \lvert y^{(m+q)}_{s}-y^{(m)}_{s} \rvert^{(2+\delta)p}e^{(2+\delta)p\int_{0}^{s}\beta_{u}du}\right]\nonumber\\ \le & 2^{(2+\delta)p-1}(1-\theta)^{(2+\delta)p}\left\{C(p)2\tilde D(p)+\sup\limits_{m \ge 1}\mathbb{E}\left[\sup\limits_{0 \le s \le T}\lvert y^{(m)}_{s} \rvert^{(2+\delta)p}e^{(2+\delta)p\int_{0}^{s}\beta_{u}du}\right]\right\} .\end{align}

	Sending $ \theta \to 1$, we obtain that
	\[ \lim\limits_{m \to +\infty}\sup\limits_{q \ge 1}\mathbb{E}\left[\sup\limits_{0 \le s \le T}\lvert y^{(m+q)}_{s}-y^{(m)}_{s} \rvert^{(2+\delta)p} e^{(2+\delta)p\int_{0}^{s}\beta_{u}du}\right] = 0,\]
	that is to say, $ \{y^{(m)}e^{\int_{0}^{\cdot} \beta_{u}du}\}_{m=0}^{+\infty} $ is a Cauchy sequence in $ \mathcal{S}^{p(2+\delta)}.$  Thus there exists $ ye^{\int_{0}^{\cdot} \beta_{u}du} \in \mathcal{S}^{p(2+\delta)}$ such that
	\[\lim\limits_{m \to +\infty}\mathbb{E}\left[ \sup\limits_{0 \le s \le T}\lvert y^{(m)}_{s}-y_{s} \rvert^{p(2+\delta)}e^{(2+\delta)p\int_{0}^{s}\beta_{u}du}  \right] = 0.\]
		By passing to a subsequence if necessary, $ y^{(m)} \to y,\, dt\times d\mathbb{P}$-a.s., and $ \sup\limits_{m \ge 0} \lvert y^{(m)}\rvert $ belongs to $ \mathcal{S}^{p(2+\delta)}(\mathbb{R}). $ 

 For any given $i=1,\,\dots,\,n,\, m,\,q \ge 1, $ utilizing It\^o's formula to
	$ \lvert y_{t}^{(m+q);i}-y_{t}^{(m);i}\rvert^{2} $, we can obtain that
	\begin{align}\label{eq25}
	&	\int_{0}^{T} \lvert z_{t}^{(m+q);i}-z_{t}^{(m);i}\rvert^{2}\ d t\nonumber \\ \le 
	&	2\sup\limits_{0 \le s \le T}\lvert  y_{s}^{(m+q);i}-y_{s}^{(m);i} \rvert  \cdot \int_{0}^{T} 
		\left(2\alpha_{s}+A\lvert y_{s}^{(m+q-1)}-y_{s}^{(m-1)} \rvert+\beta_{s} \lvert y_{s}^{(m+q);i}\rvert+\beta_{s} \lvert y_{s}^{(m);i}\rvert \right. \nonumber \\  + & \frac{\delta}{2y_{s}^{(m+q);i}}\lvert z_{s}^{(m+q);i}\rvert^{2}+\frac{\delta}{ 2y_{s}^{(m);i}} \lvert z_{s}^{(m);i}\rvert^{2}\big)\ d s -2\int_{0}^{T}\left(y_{s}^{(m+q);i}-y_{s}^{(m);i}\right)\left(z_{s}^{(m+q);i}-z_{s}^{(m);i}\right) dW_{s}.
	\end{align}
From the BDG inequality and Holder's inequality, we have
		\begin{align}\label{eq26}
	&\mathbb{E}\left[\left( \int_{0}^{T} \lvert z_{t}^{(m+q);i}-z_{t}^{(m);i}\rvert^{2}\ d t \right)^{\frac{p}{2}}\right]\nonumber \\ \le &C \mathbb{E}\left[\sup\limits_{0 \le t \le T} \lvert y_{t}^{(m+q)}-y_{t}^{(m)}\rvert^{\frac{p}{2}}\left(\int_{0}^{T} 
2\alpha_{s}+A\lvert y_{s}^{(m+q-1)}-y_{s}^{(m-1)} \rvert+\beta_{s} \lvert y_{s}^{(m+q);i}\rvert+\beta_{s} \lvert y_{s}^{(m);i}\rvert  \right.\right.\nonumber \\  + & \left.\left.\frac{\delta}{2y_{s}^{(m+q);i}}\lvert z_{s}^{(m+q);i}\rvert^{2}+\frac{\delta}{ 2y_{s}^{(m);i}} \lvert z_{s}^{(m);i}\rvert^{2} d s \right)^{\frac{p}{2}} \right]+ C \mathbb{E}\left[\sup\limits_{0 \le t \le T} \lvert y_{t}^{(m+q);i}-y_{t}^{(m);i}\rvert^{p}\right]
\nonumber \\  \le & C \mathbb{E}\left[\sup\limits_{0 \le t \le T} \lvert y_{t}^{(m+q);i}-y_{t}^{(m);i}\rvert^{p}\right]+C \left\{\mathbb{E}\left[\sup\limits_{0 \le t \le T} \lvert y_{t}^{(m+q)}-y_{t}^{(m)}\rvert^{p(2+\delta)}\right]\right\}^{\frac{1}{2(2+\delta)}} \left\{\mathbb{E}\left[\left(\int_{0}^{T} 
2\alpha_{s}\right.\right.\right.\nonumber \\ + &\left.\left.\left.A\lvert y_{s}^{(m+q-1)}-y_{s}^{(m-1)} \rvert+\beta_{s} \lvert y_{s}^{(m+q);i}\rvert+\beta_{s} \lvert y_{s}^{(m);i}\rvert  +\frac{\delta\lvert z_{s}^{(m+q);i}\rvert^{2}}{2y_{s}^{(m+q);i}}+\frac{\delta\lvert z_{s}^{(m);i}\rvert^{2} }{ 2y_{s}^{(m);i}} d s \right)^{\frac{p(2+\delta)}{3+2\delta}} \right] \right\}^{\frac{3+2\delta}{2(2+\delta)}},	\end{align}
where
\begin{align}\label{eq27}
&\mathbb{E}\left[\left(\int_{0}^{T} 
\alpha_{s} + \lvert y_{s}^{(m+q-1)}-y_{s}^{(m-1)} \rvert+\beta_{s} \lvert y_{s}^{(m+q);i}\rvert+\beta_{s} \lvert y_{s}^{(m);i}\rvert  d s \right)^{\frac{p(2+\delta)}{3+2\delta}} \right] \nonumber \\
\le &C\mathbb{E}\left[  \left(  \int_{0}^{T} \alpha_{s}d s\right)^{p}\!\!+\left(  \int_{0}^{T} \beta_{s}d s\right)^{p}\right] \!\!+C\sup_{m \ge 1}\mathbb{E}\left[ \sup_{0 \le s \le T} \lvert y_{s}^{(m)} \rvert^{p(2+\delta)}\right].
\end{align}
To estimate $  \mathbb{E}\left[\left( \int_{0}^{T}\frac{\delta}{y_{s}^{(m);i}}\lvert z_{s}^{(m);i}\rvert^{2}ds \right)^{p\frac{2+\delta}{3+2\delta}}\right]$, we define for $ \delta >0, $
\[
I(y) := 
\begin{cases}
\frac{y^{\delta+1}}{\delta(\delta+1)}-\frac{y}{\delta}-\frac{\delta}{\delta+1},\quad & y \ge 1,\\
y\ln y-y,\quad & 0< y < 1.
\end{cases}
\]
Then $ I(\cdot) \in \mathcal{C}^{2}((0,\, +\infty))$. Applying It\^o's formula to $ I(y_{s}^{(m);i}) $ and noting the fact that $ \left(\ln x\right)^{2}\mathds{1}_{0<x\le1}\le \frac{1}{x}\mathds{1}_{0<x\le1}$, we obtain 
\begin{align}\label{eq28}
&\mathbb{E}\left[\left(\int_{0}^{T}\frac{\delta\lvert z_{s}^{(m);i}\rvert^{2}}{y_{s}^{(m);i}}ds \right)^{p\frac{2+\delta}{3+2\delta}}\right] 
\nonumber \\\le &C+C\mathbb{E}\left[  \left(  \int_{0}^{T} \alpha_{s}d s\right)^{p}\!\!+\left(  \int_{0}^{T} \beta_{s}d s\right)^{p}\right] \!\!+C\sup_{m \ge 1}\mathbb{E}\left[ \sup_{0 \le s \le T} \lvert y_{s}^{(m)} \rvert^{p(2+\delta)}\right]+ C\sup_{m \ge 1}\mathbb{E}\left[\left( \int_{0}^{T}\lvert z_{s}^{(m)}\rvert^{2} ds\right)^{\frac{p}{2}}\right].\end{align}
Substituting \myeqref{eq27} and \myeqref{eq28} into \myeqref{eq26}, we have
	\begin{align}\label{eq29}
&\mathbb{E}\left[\left( \int_{0}^{T} \lvert z_{t}^{(m+q)}-z_{t}^{(m)}\rvert^{2}\ d t \right)^{\frac{p}{2}}\right]\nonumber \\ \le & C \mathbb{E}\left[\sup\limits_{0 \le t \le T} \lvert y_{t}^{(m+q)}-y_{t}^{(m)}\rvert^{p}\right]+C \left\{\mathbb{E}\left[\sup\limits_{0 \le t \le T} \lvert y_{t}^{(m+q)}-y_{t}^{(m)}\rvert^{p(2+\delta)}\right]\right\}^{\frac{1}{2(2+\delta)}},\end{align}
where $ C $ is a constant independent on $ m,\,q.$ Thus,
	\[\lim\limits_{m \to \infty} \limits\sup_{q \ge 1}\mathbb{E}\left[\left( \int_{0}^{T} \lvert z_{t}^{(m+q)}-z_{t}^{(m)}\rvert^{2}\ d t \right)^{\frac{p}{2}}\right] = 0.\]
	That is, there exists $ z \in \mathcal{M}^{p} $ such that 
	\[\lim\limits_{m \to \infty}\mathbb{E}\left[\left( \int_{0}^{T} \lvert z_{t}^{(m)}-z_{t}\rvert^{2}\ d t \right)^{\frac{p}{2}}\right] = 0.\]
	By passing to a subsequence if necessary, $ z^{(m)} \to z,\, dt\times d\mathbb{P}-$a.s., and $ \sup\limits_{m \ge 0} \lvert z^{(m)}\rvert $ belongs to $ \mathcal{M}^{p}(\mathbb{R}). $	Moreover, it follows from the Fatou's lemma,
	\begin{equation}\label{eq30}
		\mathbb{E}\left[\left(\int_{0}^{T}\frac{\delta\lvert z_{s}^{i}\rvert^{2}}{y_{s}^{i}}ds \right)^{p\frac{2+\delta}{3+2\delta}}\right] < +\infty,\,i=1,\,\dots,\,n,
	\end{equation}
 which would be used later.
	
	In view of the dominated convergence theorem for stochastic integrals, as $ m \to +\infty, $
	\[\sup\limits_{0 \le t \le T} \lvert \int_{t}^{T}\left( z_{r}^{(m);i}-z_{r}^{i}\right)\ d W_{r}\rvert \to 0\ \text{ in } \mathbb{P} .\]
	Therefore, passing to a subsequence the convergence holds \as.
	Now we follows the idea of \citep{Fan2011} to prove that passing to a subsequence,
	\[ \lim\limits_{m\to +\infty}\mathbb{E}\left[\int_{0}^{T}\lvert H_{1}^{i}(r,\,y_{r}^{(m)}) -H_{1}^{i}(r,\,y_{r})+ H_{2}^{i}(r,\,y_{r}^{(m+1);i},\,z^{(m+1);i}_{r}) - H_{2}^{i}(r,\,y_{r}^{i},\,z^{i}_{r})\rvert\ d r \right] =0. \]
	It is clear that $\lim\limits_{m\to +\infty}\mathbb{E}\left[\int_{0}^{T}\lvert H_{1}^{i}(r,\,y_{r}^{(m)}) -H_{1}^{i}(r,\,y_{r}) \rvert\ d r \right] =0$ by (H1) and the dominated convergence theorem.	
	For any $ \eta>0 $, 
	\begin{align}\label{eq31}
		&\int_{0}^{T}\lvert  H_{2}^{i}(r,\,y_{r}^{(m+1);i},\,z^{(m+1);i}_{r}) - H_{2}^{i}(r,\,y_{r}^{i},\,z^{i}_{r})\rvert\ d r \nonumber \\
		=& \int_{0}^{T}\lvert  H_{2}^{i}(r,\,y_{r}^{(m+1);i},\,z^{(m+1);i}_{r}) - H_{2}^{i}(r,\,y_{r}^{i},\,z^{i}_{r})\rvert\left(\mathds{1}_{\{\frac{1}{y_{r}^{(m+1);i} \wedge y_{r}^{i}}\le \eta\}} +\mathds{1}_{\{\frac{1}{y_{r}^{(m+1);i} \wedge y_{r}^{i}}> \eta\}}\right)\ d r.
	\end{align}
	On the one hand, 
	\begin{align}
	&\lvert H_{2}^{i}(r,\,y_{r}^{(m+1);i},\,z^{(m+1);i}_{r}) - H_{2}^{i}(r,\,y_{r}^{i},\,z^{i}_{r})\rvert\mathds{1}_{\{\frac{1}{y_{r}^{(m+1);i} \wedge y_{r}^{i}}\le \eta\}} \nonumber \\
	\le& 2\alpha_{r} + \beta_{r} \left( \sup\limits_{m \ge 1}\lvert y^{m}_{r}
	\rvert+\lvert y_{r}
	\rvert \right) +\frac{\delta \eta}{2}\left( \sup\limits_{m \ge 1}\lvert z^{(m);i}_{r}\rvert^{2} + \lvert z^{i}_{r}\rvert^{2}\right),
	\end{align}
	which is $ dt\times d\mathbb{P}-$ integrable, thus for any $ \eta >0, $
		\[ \lim\limits_{m\to +\infty}\mathbb{E}\left[\int_{0}^{T}\lvert  H_{2}^{i}(r,\,y_{r}^{(m+1);i},\,z^{(m+1);i}_{r}) - H_{2}^{i}(r,\,y_{r}^{i},\,z^{i}_{r})\rvert\mathds{1}_{\{\frac{1}{y_{r}^{(m+1);i} \wedge y_{r}^{i}}\le \eta\}}\ d r \right] =0. \]
		On the other hand, in view of (H2) and \myeqref{eq30}, we denote $ l := p\frac{2+\delta}{3+2\delta},\,k=\frac{l}{l-1} $, and obtain
	\begin{align}\label{eq32}
	&\mathbb{E}\left[\int_{0}^{T}\lvert  H_{2}^{i}(r,\,y_{r}^{(m+1);i},\,z^{(m+1);i}_{r}) - H_{2}^{i}(r,\,y_{r}^{i},\,z^{i}_{r})\rvert\mathds{1}_{\{\frac{1}{y_{r}^{(m+1);i} \wedge y_{r}^{i}}> \eta\}}\ d r \right] \nonumber \\
    \le& \left\{\mathbb{E}\left[\mathds{1}_{\{\sup\limits_{0 \le r \le T}\frac{1}{y_{r}^{(m+1);i} \wedge y_{r}^{i}}>\eta\}} \right] \right\}^{\frac{1}{k}} \left\{\mathbb{E}\left[\left( \int_{0}^{T}\lvert  H_{2}^{i}(r,\,y_{r}^{(m+1);i},\,z^{(m+1);i}_{r}) - H_{2}^{i}(r,\,y_{r}^{i},\,z^{i}_{r})\rvert\ d r\right)^{l}     \right] \right\}^{\frac{1}{l}} \nonumber \\
    \le& \eta^{-\frac{1}{k}}\!\!\left\{\mathbb{E}\left[\sup\limits_{0 \le r \le T}\frac{1}{y_{r}^{(m+1);i} \wedge y_{r}^{i}} \right] \right\}^{\frac{1}{k}}\!\! \left\{\mathbb{E}\left[\left( \int_{0}^{T} 2\alpha_{r} +\beta_{r}\left(\lvert y_{r}^{(m+1)}\rvert +\lvert y_{r}\rvert\right)+\frac{\delta\lvert z^{(m+1);i}_{r}\rvert^{2}}{2y_{r}^{(m+1);i}}+\frac{\delta\lvert z^{i}_{r}\rvert^{2}}{2y_{r}^{i}}\ d r\right)^{l}     \right] \right\}^{\frac{1}{l}}\nonumber \\
    \le& C\eta^{-\frac{1}{k}}\left\{\mathbb{E}\left[\sup\limits_{0 \le r \le T}  \frac{1}{\mathbb{E}_{r}\left[ \xi^{i} \right]} \right] \right\}^{\frac{1}{k}} \nonumber \\
       \le&   C\eta^{-\frac{1}{k}}\left\{\mathbb{E}\left[\sup\limits_{0 \le r \le T} \left[ \mathbb{E}_{r}\left[ \frac{1}{ \xi^{i} } \right]\right]^{q}\right]+1 \right\}^{\frac{1}{k}} \nonumber \\
       \le&   C\eta^{-\frac{1}{k}}\left\{\mathbb{E}\left[ \left(\frac{1}{ \xi^{i} } \right)^{q}\right]+1 \right\}^{\frac{1}{k}}   ,
	\end{align}
	where $q > 1$ is defined in (H2), and $ C>0 $ is indepedent on $ m $. Therefore, by taking a large enough $ \eta $, we can obtain that
	\[\lim\limits_{m\to +\infty}\mathbb{E}\left[\int_{0}^{T}\lvert  H_{2}^{i}(r,\,y_{r}^{(m+1);i},\,z^{(m+1);i}_{r}) - H_{2}^{i}(r,\,y_{r}^{i},\,z^{i}_{r}) \rvert \ d r\right] = 0.\]
	
  Taking limit in \myeqref{eq5}, we get $ (y,\,z) \in \mathcal{S}^{p(2+\delta)}(\mathbb{R}^{n}) \times \mathcal{M}^{p}(\mathbb{R}^{n\times d}) $ is a solution to \myeqref{eq1}. Now it remains to prove the uniqueness.
  
  Assume that $ (\hat{y},\,\hat{z}) \in  \mathcal{S}^{p(2+\delta)}(\mathbb{R}^{n}) \times \mathcal{M}^{p}(\mathbb{R}^{n\times d})$ is another solution to \myeqref{eq1}.	For any given $\theta \in (0,\, 1) $, set
	\[
	\Delta_{\theta}U_{t}:=\frac{y_{t}-\theta \hat y_{t}}{1-\theta}, \  \Delta_{\theta} V_{t} := \frac{z_{t}-\theta \hat z_{t}}{1-\theta},  
	\] 
	and
	\[
	\Delta_{\theta}\tilde U_{t}:=\frac{\hat y_{t}-\theta  y_{t}}{1-\theta}, \  \Delta_{\theta} \tilde V_{t} := \frac{\hat z_{t}-\theta  z_{t}}{1-\theta}.
	\]

		Therfore, $(\Delta_{\theta}U,\,\Delta_{\theta} V)  $ is a solution to the following system of BSDEs
	\begin{align*}
	\Delta_{\theta}U^{i}_{t} = \xi^{i} &+ \int_{t}^{T} \frac{1}{1-\theta}\Big[H_{1}^{i}\left(u,\, y_{u}\right)-\theta H_{1}^{i}\left(u,\,\hat y_{u}\right)\nonumber\\&+H_{2}^{i}\left(u,\, y_{u}^{i},\, z_{u}^{i}\right)-\theta H_{2}^{i}\left(u,\, \hat y_{u}^{i},\, \hat z_{u}^{i}\right)\Big]du - \int_{t}^{T} \Delta_{\theta} V^{i}_{u}d W_{u}.
	\end{align*}	
	It follows from (H1) that
	\begin{align}\label{eq33}
	&\frac{\mathds{1}_{\Delta_{\theta}U^{i}_{u}>0}}{1-\theta}\left[H_{1}^{i}\left(u,\, y_{u}\right)-\theta H_{1}^{i}\left(u,\, \hat y_{u}\right)+H_{2}^{i}\left(u,\, y_{u}^{i},\, z_{u}^{i}\right)-\theta H_{2}^{i}\left(u,\, \hat y_{u}^{i},\, \hat z_{u}^{i}\right)\right] \nonumber \\
	\le &\frac{A}{1-\theta}\lvert y_{u}- \hat y_{u}\rvert+H^{i}_{1}\left(u,\, \hat y_{u}\right)+ \mathds{1}_{\Delta_{\theta}U^{i}_{u}>0}H_{2}^{i}\left(u,\, 	\Delta_{\theta}U^{i}_{u},\, 	\Delta_{\theta}V^{i}_{u}\right)\nonumber \\
	\le &A\left[\lvert \Delta_{\theta}U_{u}\rvert +1+2\lvert  \hat y_{u}\rvert\right]+\alpha_{u}+\beta_{u}(\Delta_{\theta}U^{i}_{u})^{+}+\frac{\delta\mathds{1}_{\Delta_{\theta}U^{i}_{u}>0}}{2\Delta_{\theta}U^{i}_{u}}\lvert\Delta_{\theta}V^{i}_{u}\rvert^{2}.
	\end{align}	
	Applying Tanaka’s formula and It\^o's formula to $ ((\Delta_{\theta}U^{i}_{t})^{+})^{2+\delta}e^{(1+\delta)t+(2+\delta)\int_{0}^{t}\beta_{u}du} $ and using \myeqref{eq33}, we have for $ t \in [0,\,T], $
	\begin{align*}
	&((\Delta_{\theta}U^{i}_{t})^{+})^{2+\delta}e^{(1+\delta)t+(2+\delta)\int_{0}^{t}\beta_{u}du} \nonumber \\ \le & \mathbb{E}_{t}\!\left[ \left(\xi^{i}\right)^{2+\delta}e^{(1+\delta)T+(2+\delta)\int_{0}^{T}\beta_{u}du}+\!\! \int_{t}^{T}\!e^{(1+\delta)s+(2+\delta)\int_{0}^{s}\beta_{u}du}\left(A+\alpha_{s}+2A\lvert \hat y_{s}\rvert+A\lvert \Delta_{\theta}U_{s}\rvert\right)^{2+\delta}d s\right]. 
	\end{align*} 
	Analogously, 
	\begin{align*}
	&((\Delta_{\theta}\tilde{U}^{i}_{t})^{+})^{2+\delta}e^{(1+\delta)t+(2+\delta)\int_{0}^{t}\beta_{u}du} \nonumber \\ \le & \mathbb{E}_{t}\!\left[ \left(\xi^{i}\right)^{2+\delta}e^{(1+\delta)T+(2+\delta)\int_{0}^{T}\beta_{u}du}+\!\! \int_{t}^{T}\!e^{(1+\delta)s+(2+\delta)\int_{0}^{s}\beta_{u}du}\left(A+\alpha_{s}+2A\lvert  y_{s}\rvert+A\lvert \Delta_{\theta}\tilde{U}_{s}\rvert\right)^{2+\delta}d s\right]. 
	\end{align*} 
	Note the facts that
	\[\left(\Delta_{\theta}U^{i}_{t} \right)^{-}  = \frac{\left(\theta \hat y_{t}^{i} -  y^{i}_{t}\right)^{+}}{1-\theta} \le \left(\Delta_{\theta}\tilde{U}^{i}_{t} \right)^{+},\]
	and \[\left(\Delta_{\theta}\tilde{U}^{i}_{t} \right)^{-}  \le \left(\Delta_{\theta}U^{i}_{t} \right)^{+}. \]
	Thus,
	\[\lvert\Delta_{\theta}U_{t} \rvert^{2+\delta},\,\lvert\Delta_{\theta}\tilde{U}_{t} \rvert^{2+\delta} \nonumber \le n^{\frac{2+\delta}{2}}2^{1+\delta}\left(\max_{i}((\Delta_{\theta}U^{i}_{t} )^{+})^{2+\delta}+\max_{i}((\Delta_{\theta}\tilde{U}^{i}_{t} )^{+})^{2+\delta}\right).\]
	Consequently, we have
	\begin{align}\label{eq34}
	&\left( \lvert\Delta_{\theta}U_{t} \rvert ^{2+\delta}+ \lvert\Delta_{\theta}\tilde{U}_{t} \rvert ^{2+\delta} \right)e^{(1+\delta)t+(2+\delta)\int_{0}^{t}\beta_{u}du} \nonumber \\ \le &
	n^{\frac{2+\delta}{2}}2^{3+\delta}\mathbb{E}_{t}\left[ \lvert\xi\rvert^{2+\delta}e^{(1+\delta)T+(2+\delta)\int_{0}^{T}\beta_{u}du}\right]\nonumber \\ + &
	n^{\frac{2+\delta}{2}}2^{5+3\delta}\mathbb{E}_{t}\left[\int_{t}^{T}e^{(1+\delta)s+(2+\delta)\int_{0}^{s}\beta_{u}du}\left(A+\alpha_{s}\right)^{2+\delta}d s\right]\nonumber \\ + &
	n^{\frac{2+\delta}{2}}2^{6+4\delta}A^{2+\delta}\mathbb{E}_{t}\left[\int_{t}^{T}e^{(1+\delta)s+(2+\delta)\int_{0}^{s}\beta_{u}du}\left(\lvert  y_{s}\rvert^{2+\delta}+\lvert \hat y_{s}\rvert^{2+\delta}\right)d s\right]\nonumber \\ + &
	n^{\frac{2+\delta}{2}}2^{3+2\delta}A^{2+\delta}\mathbb{E}_{t}\left[\int_{t}^{T}e^{(1+\delta)s+(2+\delta)\int_{0}^{s}\beta_{u}du}\left(\lvert \Delta_{\theta}\tilde{U}_{s}\rvert^{2+\delta}+\lvert \Delta_{\theta}U_{s}\rvert^{2+\delta}\right)d s\right].
	\end{align}
	Therefore, Doob's $ L^{p} $ inequality gives for any $ t \in [0,\,T], $
	\begin{align}\label{eq35}
	&\mathbb{E}\left[\sup_{t \le s \le T}  \left( \lvert\Delta_{\theta}U_{s} \rvert ^{2+\delta}+ \lvert\Delta_{\theta}\tilde{U}_{s} \rvert ^{2+\delta} \right)^{p}e^{(1+\delta)sp+(2+\delta)p\int_{0}^{s}\beta_{u}du}  \right]\nonumber \\ \le & (\frac{p}{p-1})^{p}4^{p-1}	n^{\frac{(2+\delta)p}{2}}2^{(5+3\delta)p}\mathbb{E}\left[\left( \lvert\xi\rvert^{(2+\delta)p}+\left( \int_{0}^{T}\left(A+\alpha_{s}\right)^{2+\delta}d s\right)^{p}\right)e^{(1+\delta)pT+(2+\delta)p\int_{0}^{T}\beta_{u}du}\right]\nonumber \\ + &
	(\frac{p}{p-1})^{p}4^{p-1}n^{\frac{(2+\delta)p}{2}}2^{(6+4\delta)p}A^{(2+\delta)p}T^{p}\mathbb{E}\left[\sup_{0 \le s \le T}e^{(1+\delta)sp+(2+\delta)p\int_{0}^{s}\beta_{u}du}\left( \lvert y_{s} \rvert^{(2+\delta)p}+\lvert \hat y_{s} \rvert^{(2+\delta)p}\right) \right]\nonumber \\ + &
	(\frac{p}{p-1})^{p}4^{p-1}n^{\frac{(2+\delta)p}{2}}2^{(3+2\delta)p}A^{(2+\delta)p}(T-t)^{p}\nonumber \\ \cdot &\mathbb{E}\left[\sup_{t \le s \le T}e^{(1+\delta)ps+(2+\delta)p\int_{0}^{s}\beta_{u}du}\left(\lvert \Delta_{\theta}\tilde{U}_{s}\rvert^{2+\delta}+\lvert \Delta_{\theta}U_{s}\rvert^{2+\delta}\right)^{p}\right]\nonumber \\ = &
K(p) + \tilde B(t,\,p)\mathbb{E}\left[\sup_{t \le s \le T}e^{(1+\delta)ps+(2+\delta)p\int_{0}^{s}\beta_{u}du}\left(\lvert \Delta_{\theta}\tilde{U}_{s}\rvert^{2+\delta}+\lvert \Delta_{\theta}U_{s}\rvert^{2+\delta}\right)^{p}\right],
	\end{align}
	where $ K(p) $ is a constant independent on $\theta,\,  \tilde B(t,\,p) $ is defined in \myeqref{eq20}. 

	If $ A = 0 , $	it is evident that for any $ t \in [0,\,T], $
	\[\mathbb{E}\left[\sup_{t \le s \le T}\left( \lvert\Delta_{\theta}U_{s} \rvert ^{2+\delta}+ \lvert\Delta_{\theta}\tilde{U}_{s} \rvert ^{2+\delta} \right)^{p}e^{(1+\delta)sp+(2+\delta)p\int_{0}^{s}\beta_{u}du}  \right] \le
K(p). \]
Therefore, \[\mathbb{E}\left[ \sup\limits_{0 \le t \le T}\lvert y_{t}- \hat y_{t} \rvert^{(2+\delta)p}  \right]  \le
(1-\theta)^{(2+\delta)p}2^{(2+\delta)p-1} \left[K(p)+\mathbb{E}\left[ \sup\limits_{0 \le t \le T}\lvert \hat y_{t}\rvert^{(2+\delta)p}\right] \right]. \]
	Sending $ \theta \to 1$, we get $ y = \hat y $ and $ z= \hat z $ on $[0,\,T].  $ 
	If 	$ A > 0 , $ set $ \tilde \varepsilon = \frac{1}{\frac{p}{p-1}4^{\frac{p-1}{p}}2^{3+2\delta+\frac{1}{p}}	A^{2+\delta}n^{\frac{2+\delta}{2}} } $. Then $\tilde \varepsilon > 0. $ Let $ \tilde m_{0} $ be the unique integer such that
	\begin{equation}\label{eq36}
	\frac{T}{\tilde \varepsilon} \le \tilde m_{0} < \frac{T}{\tilde \varepsilon}  +1 .
	\end{equation} 
	
	If $ m_{0} = 1, $ that is, $ 2\tilde B(t,p)\le 2\left(\frac{p}{p-1}\right)^{p}n^{\frac{(2+\delta)p}{2}}4^{p-1}2^{(3+2\delta)p}A^{(2+\delta)p}T^{p} = \left(\frac{T}{\tilde \varepsilon} \right) ^{p}\le 1. $ Therefore, it follows from \myeqref{eq35} that 
	\begin{equation}\label{eq37}
	\mathbb{E}\left[\sup_{0 \le s \le T}\lvert\Delta_{\theta}U_{s} \rvert^{(2+\delta)p} \right]  \le 2K(p) .
	\end{equation} 
		Then we have \[\mathbb{E}\left[ \sup\limits_{0 \le t \le T}\lvert y_{t}- \hat y_{t} \rvert^{(2+\delta)p}  \right]  \le
	(1-\theta)^{(2+\delta)p}2^{(2+\delta)p-1} \left[2K(p)+\mathbb{E}\left[ \sup\limits_{0 \le t \le T}\lvert \hat y_{t}\rvert^{(2+\delta)p}\right] \right]. \]
	Sending $ \theta \to 1$, we get $ y = \hat y $ and $ z= \hat z $ on $[0,\,T]$.

	If $ m_{0} = 2$, then $ \tilde \varepsilon < T \le 2\tilde \varepsilon. $
	Thus for all $ t \in \left[T-\tilde\varepsilon,\,T\right],\, 2\tilde B(t,p)  = \left(\frac{T-t}{\tilde\varepsilon} \right) ^{p}  \le 1. $ Similarly, we have that 
	\begin{equation}\label{eq38}
	\mathbb{E}\left[\sup_{T-\tilde\varepsilon \le s \le T} \left( \lvert\Delta_{\theta}U_{s} \rvert ^{2+\delta}+ \lvert\Delta_{\theta}\tilde{U}_{s} \rvert ^{2+\delta} \right)^{p}e^{(1+\delta)sp+(2+\delta)p\int_{0}^{s}\beta_{u}du}  \right]  \le 2K(p).\end{equation} 
	In particular,
	\begin{equation}\label{eq39}
	\mathbb{E}\left[e^{(1+\delta)(T-\tilde\varepsilon)p+(2+\delta)p\int_{0}^{T-\tilde\varepsilon}\beta_{u}du}\left( \lvert\Delta_{\theta}U_{T-\tilde\varepsilon} \rvert ^{2+\delta}+ \lvert\Delta_{\theta}\tilde{U}_{T-\tilde\varepsilon} \rvert ^{2+\delta} \right)^{p} \right] \le 2K(p). \end{equation} 
	Consider the following system of BSDEs for any $0 \le t \le T-\tilde\varepsilon,$
	\begin{align*}
	\Delta_{\theta}U^{i}_{t} = \Delta_{\theta}U_{T-\tilde\varepsilon}^{i} &+ \int_{t}^{T-\tilde\varepsilon} \frac{1}{1-\theta}\Big[H_{1}^{i}\left(u,\, y_{u}\right)-\theta H_{1}^{i}\left(u,\,\hat y_{u}\right)\nonumber\\
	&+H_{2}^{i}\left(u,\, y_{u}^{i},\, z_{u}^{i}\right)-\theta H_{2}^{i}\left(u,\, \hat y_{u}^{i},\, \hat z_{u}^{i}\right)\Big]du - \int_{t}^{T-\tilde\varepsilon} \Delta_{\theta} V^{i}_{u}d W_{u},
	\end{align*}
	where $i = 1,\,\dots,\, n$. Noting the fact that $ T-\tilde\varepsilon < \tilde\varepsilon $, we obtain an inequality similar to \myeqref{eq7} for any $t \in [0,\,T-\tilde\varepsilon], $
	\begin{align} \label{eq40}
&\mathbb{E}\left[\sup_{t \le s \le T-\tilde\varepsilon}  \left( \lvert\Delta_{\theta}U_{s} \rvert ^{2+\delta}+ \lvert\Delta_{\theta}\tilde{U}_{s} \rvert ^{2+\delta} \right)^{p}e^{(1+\delta)sp+(2+\delta)p\int_{0}^{s}\beta_{u}du}  \right]\nonumber \\ \le & (\frac{p}{p-1})^{p}4^{p-1}	n^{\frac{(2+\delta)p}{2}}2^{(2+\delta)p}\mathbb{E}\left[\left(  \lvert\Delta_{\theta}U_{T-\tilde\varepsilon} \rvert ^{2+\delta}+ \lvert\Delta_{\theta}\tilde{U}_{T-\tilde\varepsilon} \rvert ^{2+\delta} \right)^{p}e^{(1+\delta)pT+(2+\delta)p\int_{0}^{T}\beta_{u}du}\right]\nonumber \\
 + &(\frac{p}{p-1})^{p}4^{p-1}	n^{\frac{(2+\delta)p}{2}}2^{(5+3\delta)p}\mathbb{E}\left[\left( \int_{0}^{T}\left(A+\alpha_{s}\right)^{2+\delta}d s\right)^{p}e^{(1+\delta)pT+(2+\delta)p\int_{0}^{T}\beta_{u}du}\right]\nonumber \\ + &
(\frac{p}{p-1})^{p}4^{p-1}n^{\frac{(2+\delta)p}{2}}2^{(6+4\delta)p}A^{(2+\delta)p}T^{p}\mathbb{E}\left[\sup_{0 \le s \le T}e^{(1+\delta)sp+(2+\delta)p\int_{0}^{s}\beta_{u}du}\left( \lvert y_{s} \rvert^{(2+\delta)p}+\lvert \hat y_{s} \rvert^{(2+\delta)p}\right) \right]\nonumber \\ + &
(\frac{p}{p-1})^{p}4^{p-1}n^{\frac{(2+\delta)p}{2}}2^{(3+2\delta)p}A^{(2+\delta)p}(T-\tilde\varepsilon-t)^{p}\nonumber \\ \cdot &\mathbb{E}\left[\sup_{t \le s \le T-\tilde\varepsilon}e^{(1+\delta)ps+(2+\delta)p\int_{0}^{s}\beta_{u}du}\left(\lvert \Delta_{\theta}\tilde{U}_{s}\rvert^{2+\delta}+\lvert \Delta_{\theta}U_{s}\rvert^{2+\delta}\right)^{p}\right]\nonumber \\ = &
\bar K(p) + \frac{1}{2}\mathbb{E}\left[\sup_{0 \le s \le T-\tilde\varepsilon}e^{(1+\delta)ps+(2+\delta)p\int_{0}^{s}\beta_{u}du}\left(\lvert \Delta_{\theta}\tilde{U}_{s}\rvert^{2+\delta}+\lvert \Delta_{\theta}U_{s}\rvert^{2+\delta}\right)^{p}\right],
	\end{align}
	where \[
	\bar{K}(p) := \left( 1+(\frac{p}{p-1})^{p}4^{p-1}	n^{\frac{(2+\delta)p}{2}}2^{(2+\delta)p+1}\right)K(p).
\]
	Therefore,
	\begin{equation}\label{eq41}
		\mathbb{E}\left[\sup_{0 \le s \le T-\tilde\varepsilon}  \left( \lvert\Delta_{\theta}U_{s} \rvert ^{2+\delta}+ \lvert\Delta_{\theta}\tilde{U}_{s} \rvert ^{2+\delta} \right)^{p}e^{(1+\delta)sp+(2+\delta)p\int_{0}^{s}\beta_{u}du} \right]  
	\le	2\bar{K}(p). 	\end{equation}
	Combining \myeqref{eq38} and \myeqref{eq41} gives that
	\begin{equation}\label{eq42}
\mathbb{E}\left[\sup_{0 \le s \le T}  \left( \lvert\Delta_{\theta}U_{s} \rvert ^{2+\delta}+ \lvert\Delta_{\theta}\tilde{U}_{s} \rvert ^{2+\delta} \right)^{p}e^{(1+\delta)sp+(2+\delta)p\int_{0}^{s}\beta_{u}du} \right]  
\le	2\bar{K}(p). 	\end{equation}

 	Then we have  \[\mathbb{E}\left[ \sup\limits_{0 \le t \le T}\lvert y_{t}- \hat y_{t} \rvert^{(2+\delta)p}  \right]  \le
 	(1-\theta)^{(2+\delta)p}2^{(2+\delta)p-1} \left[2\bar K(p)+\mathbb{E}\left[ \sup\limits_{0 \le t \le T}\lvert \hat y_{t}\rvert^{(2+\delta)p}\right] \right]. \]

	Sending $ \theta \to 1$, we get $ y = \hat y $ and $ z= \hat z $ on $[0,\,T].  $ 
	  Repeating this process, we have
	  	\begin{equation}\label{eq44}
	 \mathbb{E}\left[\sup_{0 \le s \le T}\left( \lvert\Delta_{\theta}U_{s} \rvert ^{2+\delta}+ \lvert\Delta_{\theta}\tilde{U}_{s} \rvert ^{2+\delta} \right)^{p}e^{(1+\delta)sp+(2+\delta)p\int_{0}^{s}\beta_{u}du} \right] 
	 \le 2C(p)K(p),
	 \end{equation}	
	 where $ C(p) = 1+2^{(2+\delta)p+1}\left(\frac{p}{p-1}\right)^{p}n^{\frac{(2+\delta)p}{2}}4^{p-1}+\dots+\left(2^{(2+\delta)p+1}\left(\frac{p}{p-1}\right)^{p}n^{\frac{(2+\delta)p}{2}}4^{p-1}\right)^{\tilde m_{0}-1}.$
	 Then we have  \[\mathbb{E}\left[ \sup\limits_{0 \le t \le T}\lvert y_{t}- \hat y_{t} \rvert^{(2+\delta)p}  \right]  \le
	 (1-\theta)^{(2+\delta)p}2^{(2+\delta)p-1} \left[2C(p)K(p)+\mathbb{E}\left[ \sup\limits_{0 \le t \le T}\lvert \hat y_{t}\rvert^{(2+\delta)p}\right] \right]. \]
	 
	 Sending $ \theta \to 1$, we get $ y = \hat y $ and $ z= \hat z $ on $[0,\,T].  $ 
	 This completes the proof.
\end{proof}

\section{Applications in Regime-Switching Investment Decision-Making}
In a regime-switching market, a Markov chain is used to reflect the status of the market, such as bull or bear markets. There has been extensive research on optimal consumption-investment problems and mean-variance portfolio selection problems in regime-switching markets, see, e.g. \citep{Bauerle2004, Yin2004, Hu2022}.  Recently, \citep{Hu2025} studied optimal consumption–investment problems in a regime switching market with parameters depending on both a Markov chain and a Brownian motion. Following this framework, this paper generalizes the investment model of \citep{Bahlali2018} to a regime switching market with random coefficients. In this chapter, we demonstrate that the solution to the singular system of equations established in the previous chapter represents the value function of a class of optimal investment decision problems.

We first introduce some notations. Let $\alpha_{t} $ denote a continuous-time stationary Markov chain defined on the probability space $(\Omega,\, \mathcal{F},\,\mathbb{P})$. Recall that $(\mathcal{F}_{s})_{s \in [0, \, T]}$ is the $\mathbb{P}$-completion of the filtration generated by $W$. Define the filtrations $ \tilde{\mathcal{F}}_{t}= \sigma\{W_{s},\,\alpha_{s}: 0 \le s \le t \}\vee \mathcal{N}, $ where $\mathcal N$ is the totality of all the $ \mathbb{P}$-null sets of $ \mathcal{F} $. We assume $ W_{t}$ and $\alpha_{t} $ are independent processes, and $\alpha_{t} $ takes values in a finite state space $ G = {1, 2,\,\dots,\,k} $ with $ k > 1.$ The generator of $ \alpha_{t} $ is given by $ Q = (q^{lj})_{k \times k} $ with $ q^{lj} \ge 0 $ for $ l \neq j $ and $ \sum_{j=1}^{k} q^{lj}  = 0$ for each $ l \in G. $

Consider a market with $ m $ stocks following the synamics
\[dS_{i}(t) = S_{i}(t)\left(b_{i}(t, \alpha_{t}) dt + \sum\limits_{j=1}^{d}\sigma_{ij}(t, \alpha_{t})dW^{j}_{t}\right),\,S_{i}(0) = s_{i}(0)\in \mathbb{R}_{+},\]
 where $  b_{i}(t,l) $ is the appreciation rate process and $ \sigma_{i}(t,l)=(\sigma_{i1}(t,l),\,\dots,\,\sigma_{id}(t,l)) $ is the volatility or the dispersion rate process of the $ i $th stock, corresponding to $ \alpha_{t} = l,$ for each $ i=1,\,\dots,\,m,$ and $ l \in G. $
 
 For each $ l \in G,$ define the appreciate vector $ b(t,l)=(b_{1}(t,l),\,\dots,\,b_{m}(t,l))'$, and volatility matrix
 \[
\sigma(t,l) = \begin{pmatrix}
 \sigma_{1}(t,l) \\ \sigma_{2}(t,l) \\ \sigma_{m}(t,l)
 \end{pmatrix} \equiv (\sigma_{ij}(t,l))_{m\times d}.
 \]
 We assume throughout this section that for each $ l \in G,$
 $ b(t,l) $ and $ \sigma(t,l) $ are $ \{\mathcal{F}_{t}\}_{t\ge 0} $-predictable and uniformly bounded in $ t $. We also assume the following non-degeneracy condition
 \[\sigma(t,l)\sigma(t,l)' \ge \mu I,\,\text{ for all } t \in [0,\,T]\text{ and } l \in G,\]
 holds for some constant $ \mu>0. $

 Suppose that the initial market mode $ \alpha_{0} = l_{0}. $ Consider an agent with an initial
wealth $ x > 0 $. These initial conditions are fixed throughout this section. Let us denote by $\pi_{i}(t) $  the proportion  of the agent's wealth to invest in the $ i $th stock at time $ t $. $ \pi(\cdot) = (\pi_{1}(\cdot),\,\dots,\,\pi_{m}(\cdot))'$ is called a portfolio of the agent.
 Denote by $ X(t) $ the total wealth of the agent at time $ t $. Assuming that the trading of shares takes place continuously and that transaction cost and consumption are not considered, then one has that
 \begin{align*}
 X(t) = & x+\int_{0}^{t} X(s)\pi'(s)b(s,l)d s + \int_{0}^{t} X(s)\pi'(s)\sigma(s,l)d W_{s}\nonumber \\
       =&x+\int_{0}^{t} X(s)p(s,l)\lambda(s,l) d s + \int_{0}^{t} X(s)p(s,l)d W_{s},
 \end{align*}
 where $  p(s,l) := \pi'(s)\sigma(s,l),\,\lambda(s,l) := \sigma(s,l)'(\sigma(s,l)\sigma(s,l)')^{-1}b(s,l).$
 Then $ p(s,l) $ and $\lambda(s,l)$ are valued in $ \mathbb{R}^{1\times d} $ and $ \mathbb{R}^{ d\times 1},$ respectively. Let $p(s) = (p(s,1),\,\dots,\,p(s,k))$ denote the strategy vector. The set of admissible trading strategies is defined by
 \begin{align*}
 \mathcal{A} = \{p:\Omega \times [0,\,T] \to \mathbb{R}^{1\times d}\mid 	& p \text{ is } \{\tilde{\mathcal{F}}_{t}\}_{t\in[0,T]} \text{-predictable, }\int_{0}^{T}\lvert p_{s}\rvert^{2}ds<+\infty \text{ \as},\\ &\text{ and } (X(t))_{t\in[0,T]} \in \text{class} (D).\}
 \end{align*}

 In this section, the agent's problem is to maximize his expected utility as in \citep{Bahlali2018}
 \begin{equation}\label{eq3.1}
 V(x,l) = \sup\limits_{p\in\mathcal{A}} \mathbb{E}\left[\frac{(X(T)\zeta(l))^{\gamma}}{\gamma}\right],\end{equation}
 where $\gamma \in (0,\, 1)$, and for each $l \in G,\,\zeta(l)$ is an $\mathcal{F}_{T}$-measurable random variable representing the proportion of wealth that the investor receives or pays at time $ T.$
 \begin{proposition}
 Assume $ \gamma \in (0,\,\frac{1}{3})\bigcup(\frac{1}{3},\,1).$ Further assume that there exists a constant $D>0 $ such that $ \xi = (\xi^{1},\dots,\xi^{k})' = ((\zeta(1))^{\gamma},\,\dots,\,(\zeta(k))^{\gamma})' $ satisfies $ \frac{1}{D}\le \zeta(l) \le D,\,l=1,\dots,k$. Then the value function in \myeqref{eq3.1} is given by \[V(x,l)=\frac{x^{\gamma}}{\gamma}Y^{l}_{0},\]
  where $ (Y^{l}_{t},\,Z^{l}_{t})_{l=1,\dots,k}$ is the unique solution of the multi-dimensional BSDE $ (\xi,H), $ with $H= (H^{1},\dots,H^{k})'$ is defined for $t\in [0,\,T], y=(y^{1},\dots,y^{k})'\in \mathbb{R}^{k},\,z=(z^{1},\dots,z^{k})' \in \mathbb{R}^{k\times d}$ by ($ l=1,\dots,k $)
 \[H^{l}(t,y,z) = 
 \begin{cases}
 \frac{\gamma y^{l}}{2(1-\gamma)} \lvert \lambda(t,l)' + \frac{1}{y^{l}}z^{l} \rvert^{2} + \sum_{j=1}^{k}q^{lj}y^{j}, \text{ if } y^{l} > 0,\\
 +\infty, \text{else.}                                                         
 \end{cases}\]
 Moreover, the optimal admissible strategy is given by
  \begin{equation} \label{eq3.2}
   p^{\ast}(s,l) = \frac{1}{1-\gamma}\left(\lambda(s,l)'+\frac{1}{Y(s,l)}Z(s,l)\right),\,s \in [0,\,T],\,l=1,\dots,k.
   \end{equation}
 \end{proposition}
 \begin{proof}
 	First, we prove the BSDE $ (\xi,H)$ has a unique solution in $ \mathcal{S}^{\infty}(\mathbb{R}^{k}) \times BMO(\mathbb{R}^{k\times d}).$ Let $f(x):=\max(x,\,0),\,x\in \mathbb{R}$ denote the positive part function.
We introduce a multi-dimensional BSDE: 
 	\begin{equation}\label{eq3.4}
 	y_{t}^{l}= \xi^{l}e^{q^{ll}T}+\int_{t}^{T} g^{l}(s,y_{s},z_{s}) ds - \int_{t}^{T} z_{s}^{l} d W_{s},\,t\in[0,T],\,l=1,\dots,k,
 	\end{equation}
 	where $ (g^{1},\,\dots,\,g^{k}) $ is defined for $ t\in[0,\,T],\,y \in \mathbb{R}^{k},\,z\in\mathbb{R}^{k \times d} $ by
  \[g^{l}(t,y,z) = 
 \begin{cases} 
 \frac{\gamma y^{l}}{2(1-\gamma)} \lvert \lambda(t,l)' + \frac{1}{y^{l}}z^{l} \rvert^{2} +\sum_{j=1,j\neq l}^{k}q^{lj}e^{(q^{ll}-q^{jj})t}f(y^{j}), \text{ if } y^{l} > 0,\\
 +\infty, \text{else.}                                                         
 \end{cases}\]
 Note the facts that
 \begin{itemize}
 	\item[$ \cdot $] for any $t\in[0,\,T],\,(y,\,z)\in (0, \,+\infty)\times \mathbb{R}^{1\times d},$ $ \frac{\gamma y}{2(1-\gamma)}\lvert\lambda(t,l)' + \frac{1}{y}z \rvert^{2} \le \frac{\gamma\lvert\lambda(t,l)\rvert^{2}}{1-\gamma} y + \frac{\gamma\lvert z \rvert^{2}}{(1-\gamma)y};$
 	\item[$ \cdot $] for any $t\in[0,\,T],\, \frac{\gamma y}{2(1-\gamma)}\lvert\lambda(t,l)' + \frac{1}{y}z \rvert^{2} $ is convex in $ (y,\,z)\in (0, \,+\infty)\times \mathbb{R}^{1\times d};$
    \item[$ \cdot $] for each $j= 1,\dots,k,\,j\neq l,$
    $0 \le \sum_{j=1,j\neq l}^{k}q^{lj}e^{(q^{ll}-q^{jj})t}f(y^{j}) $ is Lipschitz in $ y \in \mathbb{R}^{k} $ uniformly in $ t\in[0,T].$
 \end{itemize}
 	According to \sref{Corolary}{corollary}, \myeqref{eq3.4} admits a unique solution $ (y,\,z) \in \mathcal{S}^{\infty}(\mathbb{R}^{k}) \times BMO(\mathbb{R}^{k\times d})$ with $\delta$ defined in (H1) is given by $\delta = \frac{2\gamma}{1-\gamma}.$ Moreover, $ (y,\,z) $ is a positive solution. Applying It\^o's formula to $Y_{t}^{l} := y_{t}^{l}e^{-q^{ll}t},\,Z_{t}^{l} := z_{t}^{l}e^{-q^{ll}t},\,l = 1,\,\dots,\,k, $ we obtain that $ (Y,\,Z) \in \mathcal{S}^{\infty}(\mathbb{R}^{k}) \times BMO(\mathbb{R}^{k\times d})$  is the unique solution to the BSDE ($ \xi,\,H $).
 
 	We obtain remaining arguments by combining Theorem 3.7 in \citep{Hu2025} and Proposition 5.1 in \citep{Bahlali2018}.  Denote $ Y(t,\alpha_{t}),\,Z(t,\alpha_{t}),\,p(t,\alpha_{t}),\,\lambda(t,\alpha_{t}) $ by $Y_{t},\,Z_{t},\,p_{t},\,\lambda_{t},$ respectively. For all $ p \in \mathcal{A},$ applying It\^o's formula to $ \frac{(X(t)^{\gamma}}{\gamma}Y^{\alpha_{t}}_{t}$ and using Lemma 4.3 in \citep{Hu2020}, we have that
 	\begin{align}
 	&\frac{(X(t))^{\gamma}}{\gamma}Y^{\alpha_{t}}_{t}=\frac{x^{\gamma}}{\gamma}Y^{l_{0}}+\int_{0}^{t} (Y_{s}p_{s}+\frac{1}{\gamma}Z_{s}) dW_{s}  \nonumber \\+& \int_{0}^{t} (X(s))^{\gamma} \left(Y_{s}p_{s}\lambda_{s}+p_{s}Z_{s}'+\frac{\gamma-1}{2}Y_{s}\lvert p_{s}\rvert^{2}- \frac{Y_{s}}{2(1-\gamma)} \lvert \lambda(t,l)' + \frac{1}{Y_{s}}Z_{s}\rvert^{2}\right)  ds   
 	  \nonumber \\+&\int_{0}^{t} \frac{(X(s))^{\gamma}}{\gamma}\sum_{j,j'=1}^{k}(Y^{j}_{s}-Y^{j'}_{s}) \mathds{1}_{\{\alpha_{s-}=j'\}}d\tilde N^{j'j}_{s},
 	\end{align}
 	where $\tilde N^{j'j}_{t} = N^{j'j}_{t} - q^{j'j}t $ is the compensated Poisson martingales under the filtration $ \{\tilde{\mathcal{F}}_{t}\}_{t\in[0,T]}. $
 	
 Since $ \sup_{p \in \mathcal{A}}\left(Y_{s}p_{s}\lambda_{s}+p_{s}Z_{s}'+\frac{\gamma-1}{2}\lvert p_{s} \rvert^{2}\right)$ is obtained at $p^{\ast}_{s}= \frac{1}{1-\gamma}\left(\lambda_{s}'+\frac{1}{Y_{s}}Z_{s}\right),$ that is, $ Y_{s}p_{s}\lambda_{s}+p_{s}Z_{s}'+\frac{\gamma-1}{2}\lvert p_{s} \rvert^{2}\le \frac{Y_{s}}{2(1-\gamma)} \lvert \lambda(t,l)' + \frac{1}{Y_{s}}Z_{s}\rvert^{2},$ there exists a sequence of stopping times $\{\tau_{n}\}_{n \ge 1}$ with $ \tau_{n} \to T $ as $ n \to \infty $ such that
 	\[\mathbb{E}\left[\frac{(X(\tau_{n}))^{\gamma}}{\gamma}Y_{\tau_{n}}\right] \le \frac{x^{\gamma}}{\gamma}Y_{0}^{l_{0}}. \]
 	Applying Fatou's lemma yields that
 	\[\mathbb{E}\left[\frac{(X(T))^{\gamma}}{\gamma}\xi\right] \le \frac{x^{\gamma}}{\gamma}Y_{0}^{l_{0}}. \]  
 	If $ p^{\ast} \in \mathcal{A}, $ then the dominated convergence theorem shows that the equality holds if and only if $ p = p^{\ast}. $ It remains to prove  $ p^{\ast} \in \mathcal{A}.$ Specifically, we will show that for each fixed $l = 1,\,\dots,\,k,\, (X(t))_{t\in[0,T]} \in \text{class} (D)$ if $ p = p^{\ast}. $
 	
 Since $ \lambda(t,l) $ is uniformly bounded in $ t,$ for all $t \in [0,\,T],\,Y^{l}_{t} \ge \mathbb{E}_{t}\left[\xi^{l}\right]\ge \frac{1}{D^{\gamma}},$ and $ Z^{l}\cdot W \in BMO(\mathbb{R}^{1\times d}),$ we obtain $ p^{\ast}\cdot W\in BMO(\mathbb{R}^{1\times d}),$ and for any $ q>1, $ \[ \mathbb{E}\left[ q\int_{0}^{T}\lambda(t,l)dW_{t}-\frac{q}{2} \int_{0}^{T}\lvert \lambda(t,l) \rvert^{2} dt\right]<\infty.\]
 Therefore, $ dW^{\lambda}_{t}= dW_{t}+ \lambda(t,l)' dt$ is a new Brownnian motion under the probability measure defined by $ \frac{dQ}{dP} \mid_{ \mathcal{F}_{T}} = \mathscr{E}(-\int_{0}^{\cdot}\lambda(t,l)dW_{t})_{T},$ and
$ \int_{0}^{t} p^{\ast}(s,l) \lambda(s,l) ds+\int_{0}^{t} p^{\ast}(s,l) dW_{s} $ is a $ BMO$ martingale under $ Q. $
 
 Thus, $ X(t) = x e^{\int_{0}^{t} p^{\ast}(s,l) \lambda(s,l)  - \frac{1}{2}\lvert p^{\ast}(s,l) \rvert^{2} ds+ \int_{0}^{t} p^{\ast}(s,l) dW_{s} } = x\mathscr{E}(\int_{0}^{\cdot} p^{\ast}(s,l) \lambda(s,l) ds+\int_{0}^{\cdot} p^{\ast}(s,l) dW_{s})_{t}$ is a uniformly integrable martingale under $ Q,$ and for some $ q>1,\,\mathbb{E}^{Q}[(X(T))^{q}]<\infty$ by Theorem 3.4 in \citep{Kazamaki2006}. From Corollary 3.4 in \citep{Kazamaki2006}, we have for $ \theta >1 $ such that $ \frac{1}{\theta}+\frac{1}{q}=1$ and any stopping time $\tau \in \mathcal{T}_{0,T},$ there exists a constant $ C>0 $ such that
 $X(\tau) \le C\left( \mathbb{E}^{Q}[(X(T))^{\frac{1}{\theta}}\mid_{ \mathcal{F}_{\tau}}] \right)^{\theta}.$ Then we have 
 \begin{align*}
 	&\mathbb{E}\left[\sup\limits_{\tau \in \mathcal{T}_{0,T}}X(\tau) \right] \nonumber \\ 
 	\le & C\mathbb{E}\left[\sup\limits_{\tau \in \mathcal{T}_{0,T}}X(\tau) e^{-\frac{1}{q}\int_{0}^{T} \frac{1}{2}\lvert\lambda(s,l) \rvert^{2} ds+ \frac{1}{q}\int_{0}^{T} \lambda(s,l)  dW_{s} }   e^{-\frac{1}{q}\int_{0}^{T} \frac{1}{2}\lvert\lambda(s,l) \rvert^{2} ds- \frac{1}{q}\int_{0}^{T} \lambda(s,l)  dW_{s} }       \right] \nonumber \\
 	\le & C\left(\mathbb{E}^{Q}\left[  \sup\limits_{\tau \in \mathcal{T}_{0,T}}(X(\tau))^{q}\right]  \right)^{\frac{1}{q}}\left(\mathbb{E}\left[e^{-\frac{\theta}{q}\int_{0}^{T} \frac{1}{2}\lvert\lambda(s,l) \rvert^{2} ds+ \frac{\theta}{q}\int_{0}^{T} \lambda(s,l)  dW_{s} } \right] \right)^{\frac{1}{\theta}} \nonumber \\
 	\le & C\left( \mathbb{E}^{Q}\left[  \sup\limits_{\tau \in \mathcal{T}_{0,T}}\left( \mathbb{E}^{Q}[(X(T))^{\frac{1}{\theta}}\mid_{ \mathcal{F}_{\tau}}] \right)^{q\theta}  \right]\right)^{\frac{1}{q}} \le  C\left( \mathbb{E}^{Q}\left[(X(T))^{q}\right]\right)^{\frac{1}{q}} < \infty.
 \end{align*}
 Therefore, $ (X(t))_{t\in[0,T]} \in \text{class} (D),$ which completes the proof.
 \end{proof}

\renewcommand{\refname}

\end{document}